\documentclass{article}
\usepackage{amsmath}
\usepackage{amssymb}
\usepackage{amsthm}
\usepackage{bbm}
\usepackage{cite}
\usepackage{enumerate}
\usepackage{leftidx}
\usepackage{stmaryrd}
\usepackage{pgf}
\usepackage{tipa}

\usepackage{tikz-cd}
\usepackage{verbatim}
\usepackage{xypic}

\title{Stratified categories, geometric fixed points and a generalized Arone-Ching theorem}
\author{Saul Glasman}

\newcommand{\Ab}{\mathbf{Ab}}
\newcommand{\angs}[1]{\langle #1 \rangle}

\newcommand{\Aut}{\text{Aut}}
\newcommand{\bb}{\mathbb}
\newcommand{\bu}{\bullet}
\newcommand{\bv}{\bigvee}

\newcommand{\cof}{\text{cof}}
\newcommand{\colim}[1]{\underset{#1}{\text{colim }}}

\newcommand{\D}{\Delta}

\newcommand{\eps}{\epsilon}

\newcommand{\Fun}{\text{Fun}}

\newcommand{\Hom}{\text{Hom}}
\newcommand{\Id}{\text{Id}}
\newcommand{\id}{\text{id}}
\newcommand{\im}{\text{im }}
\newcommand{\inc}{\subseteq}
\newcommand{\inj}{\hookrightarrow}

\newcommand{\iy}{\infty}

\newcommand{\la}{\lambda}
\newcommand{\lef}{\leftarrow}
\newcommand{\La}{\Lambda}
\newcommand{\llim}[1]{\underset{#1}{\text{lim }}}
\newcommand{\mb}{\mathbf}
\newcommand{\mc}{\mathcal}
\newcommand{\mf}{\mathfrak}
\newcommand{\Map}{\text{Map}}

\newcommand{\ol}{\overline}

\newcommand{\op}{\text{op}}
\newcommand{\os}{\overset}
\newcommand{\ot}{\otimes}

\newcommand{\Si}{\Sigma}

\newcommand{\Set}{\mc{S}et}
\newcommand{\Sp}{\textbf{Sp}}

\newcommand{\td}{\widetilde}
\newcommand{\toe}{\overset{\sim} \to}
\newcommand{\twa}{\widetilde{\mathcal{O}}}
\newcommand{\Top}{\textbf{Top}}

\newcommand{\X}{\times}

\newcommand{\Su}{\bb{P}}

\newcommand{\F}{\mc{F}}
\newcommand{\Mack}{\mb{Mack}}

\newcommand{\FfG}{\F r^G}
\newcommand{\CMon}{\textbf{CMon}}

\theoremstyle{definition}

\newtheorem{cor}[subsection]{Corollary}
\newtheorem{dfn}[subsection]{Definition}
\newtheorem{exa}[subsection]{Example}

\newtheorem{lem}[subsection]{Lemma}
\newtheorem{prop}[subsection]{Proposition}
\newtheorem{rec}[subsection]{Recollection}
\newtheorem{rem}[subsection]{Remark}

\newtheorem{thm}[subsection]{Theorem}
\newtheorem{war}[subsection]{Warning}

\begin{document}

\maketitle 
\begin{abstract}
We develop a theory of Mackey functors on epiorbital categories which simultaneously generalizes the theory of genuine $G$-spectra for a finite group $G$ and the theory of $n$-excisive functors on the category of spectra. Using a new theory of stratifications of a stable $\iy$-category along a finite poset, we prove a simultaneous generalization of two reconstruction theorems: one by Abram and Kriz on recovering $G$-spectra from structure on their geometric fixed point spectra for abelian $G$, and one by Arone and Ching that recovers an $n$-excisive functor from structure on its derivatives. We deduce a strong splitting theorem for $K(n)$-local $G$-spectra and reprove a theorem of Kuhn on the $K(n)$-local splitting of Taylor towers.
\end{abstract}
\section{Introduction}

Equivariant stable homotopy theory has become notorious for its profusion of fixed point functors. The most superficially arcane of these, the \emph{geometric fixed points} first defined in \cite[\S II.9]{LMS}, also have the best formal properties. Let $G$ be a finite group and $H$ a subgroup; let $\Sp^G$ be the $\iy$-category of genuine $G$-spectra. Then the geometric $H$-fixed point spectrum, which we will regard as a functor
\[\Phi^H : \Sp^G \to \Sp,\]
is uniquely determined by the following properties:
\begin{enumerate}
\item $\Phi^H$ commutes with all homotopy colimits, and
\item $\Phi^H$ is compatible with the suspension spectrum functor in that the diagram of functors
\[\begin{tikzcd}
\Top^G \ar{r}{\Si^\iy_+} \ar{d}{(-)^H} & \Sp^G \ar{d}{\Phi^H} \\
\Top \ar{r}{\Si^\iy_+} & \Sp
\end{tikzcd}\]
commutes.
\end{enumerate}
The following additional pleasant properties follow:
\begin{enumerate}
\setcounter{enumi}{2}
\item $\Phi^H$ has a natural enhancement to a symmetric monoidal functor.
\item (Geometric fixed point Whitehead theorem) If $E$ is a $G$-spectrum such that $\Phi^H E$ is contractible for \emph{all} subgroups $H$ of $G$, then $E$ itself is contractible.
\end{enumerate}
This last property strongly suggests that it might be possible to present $G$-spectra as diagrams of their geometric fixed point spectra, much as the spectral Mackey functor approach of \cite{GM} and \cite{Bar14} employs the genuine fixed point spectra $E^H$. Such a presentation would obviously be desirable, since there are many spectra - for example, those arising as norms in the sense of \cite[\S B.5]{HHR} - whose geometric fixed points are much more accessible than their genuine fixed points. Some subtlety turns out to be required in this. For example, already for $G = C_2$, it's easy to see that the spaces of natural transformations in both directions between $\Phi^G$ and the underlying spectrum, or identity fixed point, functor $\Phi^{\{e\}}$ are contractible. Nevertheless, it can be done, and in this paper, we give a construction that accomplishes this and significantly more, drawing together work of Abram and Kriz \cite{AK13} and Arone and Ching \cite{AC15}, which itself generalizes previous work of Bauer and McCarthy \cite{BM03}.
We stress that this presentation gives a completely new model of $G$-spectra in which the fundamental data are geometric fixed point spectra, their homotopical actions by groups derived from $G$, and gluing data relating these.
To motivate this generalization, we'll draw attention to an analogy between equivariant stable homotopy theory and functor calculus, which was first pointed out to us by Mike Hopkins. For $G = C_2$ and a $G$-spectrum $E$, we have a natural cofiber sequence, the \emph{norm cofibration sequence}
\[E_{hG} \to E^G \to \Phi^G E.\]
On the other hand, suppose that $F : \Sp \to \Sp$ is a reduced 2-excisive functor in the sense of Goodwillie \cite{Goo91}. Then the Taylor tower of $F$ is simply a cofiber sequence of functors
\[D_2 F \to F \to P_1 F\]
where $D_2 F$ is the 2-homogeneous part of $F$ and $P_1 F$ is the 1-excisive approximation. Fixing a spectrum $X$, let's specialize further to the case where $E$ is the indexed smash product $X^{\wedge C_2}$ and $F$ is the functor
\[W : \Sp \to \Sp, \hspace{1 em} W(T) = (T^{\wedge C_2})^{C_2},\]
evaluated on $X$. Then the norm cofibration sequence and the Taylor tower become equivalent cofiber sequences
\[(X^{\wedge C_2})_{h C_2} \to (X^{\wedge C_2})^{C_2} \to X.\]
This equivalence can be made into the basis for an equivalence of $\iy$-categories 
\[\Sp^{C_2} \simeq \Fun^{\text{2-exc}}(\Sp, \Sp)\]
between $\Sp^{C_2}$ and the category of reduced 2-excisive functors $\Sp \to \Sp$ under which the identity fixed points correspond to the second derivative and the geometric $C_2$-fixed points correspond to the first derivative.

This coincidence suggests the existence of a systematic analogy between Goodwillie derivatives and geometric fixed points. This paper develops a common context for the Goodwillie calculus of functors between stable $\iy$-categories and equivariant stable homotopy theory - that of \emph{Mackey functors on epiorbital categories} - which makes this analogy precise. In brief, $n$-excisive functors are governed by the category $\F_s^{\leq n}$ of finite sets of cardinality at most $n$ and surjective maps in precisely the same way as $G$-spectra are governed by the category $\mc{O}_G$ of transitive $G$-sets, and the visible equivalence of categories between $\F_s^{\leq 2}$ and $\mc{O}_{C_2}$ accounts for the equivalence between $\Sp^{C_2}$ and $\Fun^{\text{2-exc}}(\Sp, \Sp)$.

A beautiful presentation of $n$-excisive functors on spectra via structure on their derivatives has been constructed by Arone and Ching in \cite{AC15}, and we extend their result to our more general context, where it provides the desired presentation of the $\iy$-category of $G$-spectra. Along the way, we develop a formalism of \emph{stratified stable $\iy$-categories} that encompasses, on the one hand, the category of Mackey functors on a epiorbital category, and on the other hand, monoidal stable $\iy$-categories equipped with a family of homological localizations, such as the category of $p$-local spectra with its chromatic filtration. Our theorem will be a special case of a general reconstruction theorem for objects of stratified stable $\iy$-categories. In the context of families of homological localizations, similar results have been obtained by Antol\'in-Camarena and Barthel \cite{A-CB14}. In upcoming work, we plan to use this presentation to give an explicit and homotopy-invariant description of the Hill-Hopkins-Ravenel norm.

The structure of the paper is as follows. In Section \ref{sec:orbmack}, we develop the theory of epiorbital categories and their Mackey functors, culminating in the statement and proof of our version, Theorem \ref{thm:ac}, of the Arone-Ching comonadicity theorem \cite[Theorem 3.13]{AC15}. In Section \ref{sec:strat}, we define stratified stable $\iy$-categories (Definition \ref{def:strat}) and give several examples, then state and prove the classification of objects of a stratified stable $\iy$-category (Theorem \ref{strat}), in effect giving a description of the category of coalgebras for the comonad of Section \ref{sec:orbmack}. We unpack the implications of our theorem for $C_p$-spectra in detail in Examples \ref{exa:Cp}, obtaining the classical description of $C_p$-spectra via the Tate fracture square. Finally, in the short Section \ref{sec:knlo}, we prove a strong and general splitting theorem reminiscent of the tom Dieck splitting, Theorem \ref{thm:knlo}, for Mackey functors valued in $K(n)$-local spectra, recovering a result of Kuhn \cite{Kuh04a} on functor calculus. We believe the $G$-spectrum case of this theorem to be new.

This project has benefited from conversations with Greg Arone, Clark Barwick, Michael Ching, Jacob Lurie, Randy McCarthy and Tomer Schlank. A significant part of the impetus for the paper came from a question posed by Mike Hopkins to Clark Barwick. We thank Clark Barwick for emphasizing the importance of the nonabelian derived category and describing how to deduce Theorem \ref{thm1} from Lemma \ref{thm1DA}. We thank David Ayala, Aaron Mazel-Gee and Nick Rozenblyum for pointing out an erroneous hypothesis in Definition \ref{def:strat}. Finally, we wish to thank all of the participants of the Bourbon Seminar for generating and sustaining an environment of such productivity week after week.

\section{Epiorbital categories and Mackey functors}\label{sec:orbmack}

Our Arone-Ching theorem will be applicable to a certain class of ``generalized equivariant homotopy theories" that encompasses the usual theory of genuine $G$-spectra and the theory of $n$-excisive functors on spectra. Such a theory springs from a ``epiorbital category" with properties similar to the category of orbits for a finite group $G$.
\begin{dfn}
A \emph{epiorbital category} or EOC is an essentially finite category $\mc{M}$ satisfying the following condtions:
\begin{itemize}
\item Every morphism in $\mc{M}$ is an epimorphism.
\item $\mc{M}$ admits pushouts and coequalizers; equivalently, $\mc{M}$ admits colimits over finite connected diagrams.
\end{itemize}
\end{dfn}
Epiorbital categories have a strong directionality. 
\begin{dfn} \label{orbposet} It follows immediately from Yoneda's lemma that any endomorphism in a epiorbital category is an isomorphism, and so the set of isomorphism classes of objects of $\mc{M}$ carries a natural partial order wherein $[X] \geq [Y]$ if and only if the morphism set $\mc{M}(X, Y)$ is nonempty. We'll call this poset $\mc{P}_\mc{M}$.
\end{dfn}
\begin{exa}
Let $G$ be a finite group. Then the \emph{orbit category} $\mc{O}_G$, which is defined as the category of sets with transitive $G$-action, is a EOC.
\end{exa}
\begin{exa}\label{exa:G/H}
If $G$ is a finite group and $H$ a subgroup, then the full subcategory
\[\mc{O}_{G/H} \inc \mc{O}_G\]
spanned by those orbits on which $H$ acts trivially is again a EOC. There's no clash of notation here, for if $H$ happens to be normal, then $\mc{O}_{G/H}$ is clearly just the orbit category of the group $G/H$.
\end{exa}
\begin{exa}\label{exa:fsn}
Let $\F_s$ be the category of finite sets and surjective maps and let $\F_s^{\leq n}$ be the full subcategory of $\F_s$ spanned by the sets of cardinality at most $n$. Then $\F_s^{\leq n}$ is a EOC.
\end{exa}
The following properties of epiorbital categories will be useful:
\begin{lem}\label{orbher}
Let $\mc{M}$ be a EOC, let $\mc{K}$ be any finite category and let $\rho : \mc{K} \to \mc{M}$ be a functor. Then the overcategory $\mc{M}_{/ \rho}$ is a EOC.
\end{lem}
\begin{proof}
Immediate.
\end{proof}
\begin{lem}\label{multifinal}
Let $\mc{M}$ be a EOC. Then any connected component of $\mc{M}$ has a final object.
\end{lem}
\begin{proof}
We'll show that any object of $\mc{M}$ whose isomorphism class is minimal with respect to the natural partial order is final in its connected component. Indeed, let $T$ be such an object and let $X$ an object in its connected component. Then there's a zigzag of morphisms
\[T \os{f_1} \leftarrow Y_1 \os{g_1} \to Y_2 \os{f_2} \leftarrow \cdots \os{f_n} \leftarrow  X.\]
By taking iterated pushouts, we can inductively replace this zigzag with a diagram
\[T \os{f} \to Y \os{g} \leftarrow X.\]
Moreover, by the choice of $T$, $f$ must be an isomorphism, so there is a morphism from $X$ to $T$. Now suppose we have a pair of morphisms
\[f, g: X \rightrightarrows T.\]
Then $h$, the coequalizer of $f$ and $g$, is a morphism with source $T$, and therefore an isomorphism. By composing with $h^{-1}$, we see that $f = g$.
\end{proof}
If $\mc{C}$ is any category, we'll denote by $\mc{C}^\amalg$ the closure of $\mc{C}$ under formal finite coproducts. By definition, $\mc{C}^\amalg$ is the full subcategory of $\Fun(\mc{C}^\op, \Set)$ spanned by the finite coproducts of representable functors. This is an operation we'll frequently want to perform on EOCs.
\begin{exa}
$\mc{O}_G^\amalg$ is equivalent to $\F^G$, the category of all finite $G$-sets, since any finite $G$-set decomposes uniquely into orbits.
\end{exa}
One very relevant property of $\F^G$ is that it's meaningful to take Mackey functors over it: $\F^G$ is \emph{disjunctive} in the sense of \cite{Bar14}, so its effective Burnside $\iy$-category $A^{eff}(\F^G)$ can be formed and admits direct sums. We'll quickly recall these ideas.
\begin{dfn}
Let $\mb{C}$ be an $\iy$-category which admits finite products and coproducts and a zero object. We'll say that $\mb{C}$ \emph{admits direct sums}, or is \emph{semiadditive}, if for every pair of objects $X, Y \in \mb{C}$, the natural map
\[X \amalg Y \to X \X Y\]
provided by the zero object is an equivalence.
\end{dfn}
\begin{rem}
It's not hard to see that each mapping space of a semiadditive $\iy$-category naturally carries the structure of a commutative monoid space; the term \emph{additive} $\iy$-category is traditionally reserved for categories whose mapping spaces are grouplike.
\end{rem}
\begin{dfn}
An $\iy$-category $\mb{C}$ is called \emph{disjunctive} if 
\begin{itemize}
\item $\mb{C}$ admits pullbacks and finite coproducts,
\item for each finite set $I$ and for each $I$-tuple $(X_i)_{i \in I}$ of objects of $\mb{C}$, the natural functor
\[\prod_i \mb{C}_{/X_i} \to \mb{C}_{/ \coprod_i X_i}\]
is an equivalence of categories.
\end{itemize}
\end{dfn}
Let $\mb{C}$ be an $\iy$-category which admits pullbacks. Then one can construct \cite[3.6]{Bar14} an $\iy$-category $A^{eff}(\mb{C})$, called the \emph{effective Burnside category of $\mb{C}$}, whose objects are those of $\mb{C}$, whose morphisms are spans
\[\begin{tikzcd}[column sep = small]
& Z \ar{dr} \ar{dl} \\
X && Y
\end{tikzcd}\]
in $\mb{C}$, and where composition is performed by forming pullbacks.
\begin{prop}{\cite[4.3]{Bar14}} If $\mb{C}$ is disjunctive, then $A^{eff}(\mb{C})$ is semiadditive.
\end{prop}
\qed
\begin{exa}
A functor from $A^{eff}(\F^G)$ to $\Ab$ which preserves direct sums is precisely a Mackey functor in the sense of \cite{tom73}.
\end{exa}
The content of the next lemma is that it's meaningful to talk about Mackey functors over arbitrary epiorbital categories:
\begin{lem} \label{lem:eocdisj}
Let $\mc{M}$ be a epiorbital category. Then $\mc{M}^\amalg$ is disjunctive.
\end{lem}
\begin{proof}
The condition that
\[\prod_i \mc{C}_{/X_i} \to \mc{C}_{/ \coprod_i X_i}\]
is an equivalence is satisfied for any $\mc{C}$ of the form $\mc{D}^\amalg$, so we only need to show that $\mc{M}^\amalg$ admits pullbacks. 

Let $\rho : \La^2_2 \to \mc{M}^\amalg$ be a diagram
\[\begin{tikzcd}
& X \ar{d} \\
Y \ar{r} & Z
\end{tikzcd} \]
in $\mc{M}^\amalg$. If any of $X$, $Y$ or $Z$ are empty then the pullback exists and is empty, so let's assume all are nonempty.

If $Z$ decomposes nontrivially as a coproduct $Z_1 \coprod Z_2$, we get a decomposition of diagrams $\rho = \rho_1 \coprod \rho_2$, and if each $\rho_i$ admits a limit $W_i$, then $W_1 \coprod W_2$ is a limit of $\rho$. Thus it suffices to assume $Z$ is representable.

On the other hand, if $X$ decomposes nontrivially as $X_1 \coprod X_2$, and if
\[\begin{tikzcd}
W_1 \ar{r} \ar{d} & X_1 \ar{d} & W_ 2 \ar{r} \ar{d} & X_2 \ar{d} \\
Y \ar{r} & Z, & Y \ar{r} & Z
\end{tikzcd} \]
are pullback diagrams, then
\[\begin{tikzcd}
W_1 \coprod W_2 \ar{r} \ar{d} & X \ar{d} \\
Y \ar{r} & Z \end{tikzcd}
\]
is a pullback diagram. After carrying out the same argument for $Y$, it's enough to assume that $X$, $Y$ and $Z$ are all representable - in other words, that $\rho$ may be lifted to a diagram $\td{\rho} : \La^2_2 \to \mc{M}$.

But it now follows from Lemma \ref{orbher} and Lemma \ref{multifinal} that $\rho$ admits a limit, since to give a limit of $\rho$ in $\mc{M}^\amalg$ is, tautologically, to give a final object in each connected component of the EOC $\mc{M}_{/ \td{\rho}}$.
\end{proof}
\begin{dfn}
There is a more general notion of \emph{orbital $\iy$-category} which features centrally in the upcoming work \cite{BDGNS}. An orbital $\iy$-category is simply any $\iy$-category $\mc{M}$ for which $\mc{M}^\amalg$ admits pullbacks, from which it follows that $\mc{M}^\amalg$ is disjunctive. By Lemma \ref{lem:eocdisj}, epiorbital categories are orbital. Epiorbital categories are the focus of this paper, but some results will be stated for general orbital $\iy$-categories.
\end{dfn}
\begin{dfn} \label{mackfun}
If $\mc{M}$ is an orbital $\iy$-category, we'll write $A^{eff}(\mc{M})$ for the effective Burnside category $A^{eff}(\mc{M}^\amalg)$; we don't expect this notation to cause confusion. $A^{eff}(\mc{M})$ is semiadditive, and if $\mb{C}$ is a semiadditive presentable $\iy$-category, then we'll denote by $\Mack(\mc{M}, \mb{C})$ the category of $\mb{C}$-valued Mackey functors on $\mc{M}$: the category of additive (i.e. direct-sum-preserving) functors from $\mc{M}$ to $\mb{C}$. If $\mb{C} = \Sp$, then we'll usually omit $\mb{C}$ and refer to the category simply as $\Mack(\mc{M})$.
\end{dfn}
\begin{exa}
When $\mc{M}$ is the orbit category $\mc{O}_G$, $\Mack(\mc{M})$ is the category of \emph{spectral Mackey functors} for $G$ \cite{Bar14}, which is a model for the homotopy theory of genuine $G$-spectra.
\end{exa}
\begin{exa}
The natural target of both the genuine fixed point functor $(-)^H$ and the geometric fixed point functor $\Phi^H$ (of which more anon) on $\Mack(\mc{O}_G)$ is $\Mack(\mc{O}_{G/H})$ (Example \ref{exa:G/H}), even when $H$ is not normal in $G$.
\end{exa}

\begin{thm}
When $\mc{M}$ is the category $\F_s^{\leq n}$ of Example \ref{exa:fsn}, $\Mack(\mc{M}, \mb{C})$ is equivalent to the category of (reduced, filtered-colimit-preserving) $n$-excisive functors from $\Sp$ to $\mb{C}$. This equivalence has the property that if $F : \Sp \to \mb{C}$ is $n$-excisive, if $E$ is the corresponding Mackey functor and if $S$ is a set, then $E(S)$ is equivalent to the $S$-indexed cross-effect of $F$ evaluated on an $S$-indexed set of spheres. In particular, if $S$ has $n$-elements, then $E(S)$ is equivalent to the $n$th derivative $\bb{D}_n F$ as spectra with $\Si_n$-action.

This equivalence is the subject of the separate paper \cite{Gla16}.
\end{thm}

The following is a significant technical lemma that provides control over the values of many universally defined Mackey functors, including the fixed points of the free genuine equivariant $G$-spectrum on a spectrum with $G$-action.
\begin{lem}
Suppose that $\mb{A}$, $\mb{B}$, $\mb{C}$ are semiadditive $\iy$-categories and $\phi : \mb{A} \to \mb{B}$, $F: \mb{A} \to \mb{C}$ are additive functors. Suppose the left Kan extension $\phi_! F : \mb{B} \to \mb{C}$ exists. Then $\phi_! F$ is additive.
\end{lem}
\begin{proof}
We must verify that $\phi_! F$ preserves zero objects and direct sums of pairs of objects. The first is obvious, so let $X, Y$ be objects of $\mb{B}$. Then
\[\phi_! F(X \oplus Y) \simeq \colim{(\phi(Z) \to X \oplus Y) \in \mb{A} \X_\mb{B} \mb{B}_{/X \oplus Y}} F(Z).\]
Let
\[a : (\mb{A} \X_\mb{B} \mb{B}_{/X}) \X (\mb{A} \X_\mb{B} \mb{B}_{/Y}) \to \mb{A} \X_\mb{B} \mb{B}_{/X \oplus Y}\]
be the functor with
\[a(\phi(Z_1) \to X, \phi(Z_2) \to Y) = (\phi(Z_1 \oplus Z_2) \simeq \phi(Z_1) \oplus \phi(Z_2) \to X \oplus Y).\]
Then we claim that $a$ is cofinal. Thus we must verify that for each object $k : \phi(Z) \to X \oplus Y$ of $\mb{A} \X_\mb{B} \mb{B}_{/X \oplus Y}$, the overcategory
\[ \mb{O} := (\mb{A} \X_\mb{B} \mb{B}_{/X}) \X (\mb{A} \X_\mb{B} \mb{B}_{/Y}) \X_{\mb{A} \X_\mb{B} \mb{B}_{/X \oplus Y}} (\mb{A} \X_\mb{B} \mb{B}_{/X \oplus Y})_{/k}\]
is weakly contractible. Indeed, we claim that $\mb{O}$ has an initial object. An object of $\mb{O}$ is a pair $(Z_1, Z_2)$ of objects of $\mb{A}$ together with a morphism $\delta : Z \to Z_1 \oplus Z_2$ and a commutative diagram
\[\begin{tikzcd}
\phi(Z) \ar{d}{\phi(\delta)} \ar{r}{k} & X \oplus Y \\
\phi(Z_1 \oplus Z_2) \ar{ru}. \end{tikzcd}\]
Then the initial object of $\mb{O}$ is evidently the diagonal map $\D : Z \to Z \oplus Z$ together with the commutative diagram
\[\begin{tikzcd}
\phi(Z) \ar{d}{\phi(\D)} \ar{r}{k} & X \oplus Y \\
\phi(Z \oplus Z) \ar{ru}[below right]{k_X \oplus k_Y}. \end{tikzcd}\]
Now there is a commutative diagram
\[\begin{tikzcd} (\mb{A} \X_\mb{B} \mb{B}_{/X}) \X (\mb{A} \X_\mb{B} \mb{B}_{/Y}) \ar{r}{(F \oplus F)} \ar{d}{a} & \mb{C} \\
 \mb{A} \X_\mb{B} \mb{B}_{/X \oplus Y} \ar{ru}[below right]{F}, \end{tikzcd}\]
allowing us, by our cofinality result, to rewrite
\[\phi_!F(X \oplus Y) \simeq \colim{(\phi(Z_1) \to X, \phi(Z_2) \to Y) \in(\mb{A} \X_\mb{B} \mb{B}_{/X}) \X (\mb{A} \X_\mb{B} \mb{B}_{/Y})} F(Z_1) \oplus F(Z_2).\]
Now for any pair of functors $b_1, b_2 : K \to \mb{C}$, we have the commutation of colimits
\[\colim{}(b_1 \oplus b_2) \simeq \colim{}(b_1) \oplus \colim{}(b_2).\]
In particular, if $b_2$ is the constant functor valued at some object $P$, then
\[\colim{}(b_1 \oplus b_2) \simeq \colim{}(b_1) \oplus (P \ot K),\]
and if $K$ is weakly contractible, then
\[\colim{}(b_1 \oplus b_2) \simeq \colim{}(b_1) \oplus P.\]
Note that $\mb{A} \X_\mb{B} \mb{B}_{/ X}$ is weakly contractible for every $X \in \mb{B}$: indeed, the essentially unique object 
\[ (\mb{0}_A, \phi(\mb{0}_A) \simeq \mb{0}_B \to X)  \]
is an initial object of $\mb{A} \X_\mb{B} \mb{B}_{/ X}$. Hence
\begin{align*}
\phi_!F(X \oplus Y) & \simeq \colim{(\phi(Z_1) \to X, \phi(Z_2) \to Y) \in(\mb{A} \X_\mb{B} \mb{B}_{/X}) \X (\mb{A} \X_\mb{B} \mb{B}_{/Y})} F(Z_1) \oplus F(Z_2) \\
& \simeq \colim{(\phi(Z_1) \to X) \in \mb{A} \X_\mb{B} \mb{B}_{/X}} \hspace{1 em} \colim{(\phi(Z_2) \to Y) \in \mb{A} \X_\mb{B} \mb{B}_{/Y}} F(Z_1) \oplus F(Z_2) \\
& \simeq \colim{(\phi(Z_1) \to X) \in \mb{A} \X_\mb{B} \mb{B}_{/X}} F(Z_1) \oplus \phi_! F(Y) \\
& \simeq \phi_! F(X) \oplus \phi_! F(Y).
\end{align*}
\end{proof}
If $\mb{I}$ and $\mb{J}$ are disjunctive $\iy$-categories and $F : \mb{I} \to \mb{J}$ is a functor which preserves pullbacks and finite coproducts, then $F$ clearly induces an additive functor 
\[A^{eff}(F) : A^{eff}(\mb{I}) \to A^{eff}(\mb{J}).\]
So if $\mb{C}$ admits limits and colimits, we get functors
\[A^{eff}(F)_!, A^{eff}(F)_* : \Mack(\mb{I}, \mb{C}) \to \Mack(\mb{J}, \mb{C}).\]

If $\mc{M}$ is an orbital $\iy$-category, we'll call a full subcategory $\mc{N}$ of $\mc{M}$ \emph{downwardly-closed} if whenever $X \in \mc{N}$ and $Y \in \mc{M}$ with $\Map(X, Y)$ nonempty, we also have $Y \in \mc{N}$. Equivalently, $\mc{N} = \phi^{-1}(\{1\})$ for some functor $\phi: \mc{M} \to \D^1$. Upwardly-closed subcategories are defined dually.

There are plenty of examples of these: if $\mc{M}$ is epiorbital, then for any $X \in \mc{M}$, the full subcategory of $\mc{M}_{\leq X} \inc \mc{M}$ spanned by those objects $Y$ with $[Y] \leq [X]$ is downwardly-closed. Note also that any downwardly-closed subcategory of an orbital $\iy$-category is itself orbital, and any downwardly-closed subcategory of an EOC is itself a EOC.

For the following few lemmas, we'll let $\mc{M}$ be orbital, let $\mc{N}$ be a downwardly-closed subcategory of $\mc{M}$ and let $\mc{T}$ be its upwardly-closed complement.

\begin{lem} \label{lem:jret}  The inclusion $i_{\mc{N}} : \mc{N}_\amalg \to \mc{M}_\amalg$ admits a canonical retraction $j_{\mc{N}}$ which is right adjoint to $i_{\mc{N}}$. \end{lem}
\begin{proof}
By definition, $\mc{M}_\amalg$ is the full subcategory of $\Fun(\mc{M}^\op, \Top)$ spanned by coproducts of representables, and $i_{\mc{N}}$ is given by left Kan extension. But since the value of the functor represented by an object of $\mc{T}$ on any object of $\mc{N}$ is empty, the restriction $\Fun(\mc{M}^\op, \Set)$ to $\Fun(\mc{N}^\op, \Set)$ preserves coproducts of representables. This is the desired retraction. (Note that orbitality of $\mc{M}$ is not required for this lemma to hold.)
\end{proof}
We should think of $j_{\mc{N}}$ as ``formally set all objects of $\mc{T}$ to $\emptyset$". We note that $j_{\mc{N}}$ preserves coproducts.

\begin{dfn}
We define the \emph{geometric value at $\mc{N}$} functor 
\[\Phi^\mc{N}: \Mack(\mc{M}, \mb{C}) \to \Mack(\mc{N}, \mb{C})\]
 by the left Kan extension $(A^{eff}(j_{\mc{N}}))_!$.

The right adjoint of $\Phi^\mc{N}$, the \emph{extension by zero from $\mc{N}$} functor, will be denoted $\Xi^{\mc{N}}$. For an object $X$ of $\mc{M}$, we'll denote $\Phi^{\mc{M}_{\leq X}}$ by $\Phi^X$. If $E$ is an object of $\Mack(\mc{M}, \mb{C})$, then we'll write $E^{\Phi X}$ for the value of $\Phi^X E$ on $X$, the ``geometric fixed point spectrum at $X$".
\end{dfn}

\begin{exa}
Let $G$ be a finite group and let $H \leq G$ be a subgroup. If $\mc{M} = \mc{O}_G$, then $\Phi^{G/H}$ is the classical functor of geometric fixed points \cite[Example B.6]{Bar14}. Here the usual notation would be $\Phi^H$, not $\Phi^{G/H}$, and we apologize for the clash.

We have an equivalence
\[(\mc{O}_G)_{\leq G/H} \simeq \mc{O}_{G/H}\]
and so $\Phi^{G/H}$ naturally takes values in $\Mack(\mc{O}_{G/H})$, as previously claimed.
\end{exa}

\begin{exa}
If $\mc{M} = \F_s^{\leq n}$, and $k < n$, then $\Xi^{\F_s^{\leq k}}$ is the functor which regards a $k$-excisive functor as an $n$-excisive functor, and its left adjoint $\Phi^{\F_s^{\leq k}}$ is the $k$-excisive approximation functor.
\end{exa}

\begin{dfn}
Using the notation of Lemma \ref{lem:jret}, let $A^{eff}(\mc{T})$ be the effective Burnside category of $(\mc{T})^\amalg$, or equivalently, the full subcategory of $A^{eff}(\mc{M})$ spanned by the objects of $(\mc{T})^\amalg$. Let $A^{eff}(i_{\mc{T}})$ be the inclusion of this full subcategory.
\end{dfn}

\begin{rem}
Let $\mc{G}$ be a groupoid. We can form the effective Burnside category $A^{eff}(\mc{G})$, since any commutative square in a groupoid is a pullback square. Moreover, there's a natural equivalence of $\iy$-categories
\[c_{\mc{G}} : A^{eff}(\mc{G}) \toe \mc{G}\]
which maps the span 
\[x \os{g} \leftarrow y \os{h} \to z\]
to the morphism $hg^{-1} : x \to z$. We'll sometimes implicitly invoke this equivalence.
\end{rem}
If $X$ is an object of $\mc{M}$, let $\mc{G}_X$ be the groupoid spanned by the isomorphism class of $X$. Form the effective Burnside category $A^{eff}(\mc{G}_X^\amalg)$ and let $i_X : A^{eff}(\mc{G}_X) \to A^{eff}(\mc{M}_{\leq X})$ be the inclusion. It follows from the directionality of $\mc{M}_{\leq X}$ that $i_X$ is fully faithful.

Since there's no room for meaningful transfer maps, we might guess that a Mackey functor on a group $\mc{G}$ contains no more information than an object with $\mc{G}$-action. This is indeed the case, but the proof is technical and so we defer the bulk of it to Appendix \ref{app:Gadd}.

\begin{thm} \label{thm:Gadd}
Let $\mc{G}$ be a groupoid. Then $A^{eff}(\mc{G}^\amalg)$ is the free semiadditive $\iy$-category on $\mc{G}$: for any semiadditive $\iy$-category $\mb{C}$, the natural inclusion induces an equivalence of categories
\[\Fun^\oplus(A^{eff}(\mc{G}^\amalg), \mb{C}) \to \Fun(\mc{G}, \mb{C}).\]
\end{thm}
We'll use this equivalence implicitly from now on.

The class of Mackey functors left or right Kan extended from groupoids is an interesting one. For instance, let $G$ be a finite group and let $X$ be a spectrum with $G$-action, which by Theorem \ref{thm:Gadd} we may regard as an object of $\Mack(\mc{G}_{G/e})$. Denote by $i : \mc{G}_{G/e}^\amalg \inj \F_G$ the inclusion. Then the left Kan extension of $X$ along $A^{eff}(i)$ is the \emph{free genuine $G$-spectrum on $X$}, often denoted by
\[EG_+ \wedge X.\]
Similarly, the right Kan extension is the \emph{cofree genuine $G$-spectrum on $X$}, also known as
\[F(EG_+, X).\]
We know from other models of $G$-spectra that for any subgroup $H \leq G$,
\[(EG_+ \wedge X)^H \simeq X_{hH}, \, F(EG_+, X)^H \simeq X^{hH}.\]
It would be desirable, however, to have a proof of these facts internal to our framework. The following lemma is a more general version of this result:

\begin{lem}
Suppose $\mc{M}$ is an orbital $\iy$-category and $i : \mc{T} \inj \mc{M}$ is the inclusion of an upwardly closed subcategory. Let $\mb{C}$ be a semiadditive $\iy$-category with all colimits, and let $B \in \Mack(\mc{T}, \mb{C})$ be a Mackey functor. Then for any $Y \in \mc{M}$,
\[i_! B(Y) \simeq \colim{\mc{M}^\amalg_{/Y} \X_{\mc{M}^\amalg} \mc{T}^\amalg} B.\]
\end{lem}
\begin{proof}
For now, we content ourselves with a sketched proof and leave the details as an exercise. This result will be assumed in some examples in Section \ref{sec:strat} but will not feature integrally in the results of the paper.

For the sake of avoiding the ambiguity that can precipitate from the use of overcategory notation, we clarify our notation. Write
\[\mc{T}^\amalg_{/Y} = \mc{M}^\amalg_{/Y} \X_{\mc{M}^\amalg} \mc{T}^\amalg\]
and
\[ (A^{eff}(\mc{T}))_{/Y} = (A^{eff}(\mc{M}))_{/Y} \X_{A^{eff}(\mc{M})} A^{eff}(\mc{T}).\]
Then we have
\[i_!B(Y) \simeq \colim{(A^{eff}(\mc{T}))_{/Y}} B.\]
The objects of $(A^{eff}(\mc{T}))_{/Y}$ are of the form
\[\begin{tikzcd}[column sep = tiny] & T' \ar{dl}[above]{f} \ar{dr} \\ T && Y. \end{tikzcd}\]
We can define a functor $V : \mc{T}^\amalg_{/Y} \to (A^{eff}(\mc{T}))_{/Y}$ by
\[V(T \to Y) = \left( \begin{tikzcd}[column sep = tiny] & T \ar[equals]{dl} \ar{dr} \\ T && Y \end{tikzcd} \right).\]
The essential image of $V$ comprises those diagrams for which $f : T' \to T$ is an equivalence, and it can be shown that $V$ is homotopic to the inclusion of the full subcategory $(\ol{A^{eff}(\mc{M})})_{/Y}$ spanned by such diagrams. Moreover, the inclusion of $(\ol{A^{eff}(\mc{M})})_{/Y}$ has a left adjoint $\phi$ with
\[\phi \left( \begin{tikzcd}[column sep = tiny] & T' \ar{dl} \ar{dr} \\ T && Y \end{tikzcd} \right)  \simeq \left( \begin{tikzcd}[column sep = tiny] & T' \ar[equals]{dl} \ar{dr} \\ T' && Y \end{tikzcd} \right). \]
Thus this inclusion is cofinal, which yields the result.
\end{proof}

\begin{cor}
Suppose $\mc{M}$ is an epiorbital category and $X \in \mc{M}$ is a maximal object. Let 
\[i_X : A^{eff}(\mc{G}_X) \to A^{eff}(\mc{M})\]
be the inclusion, let $\mb{C}$ be a semiadditive $\iy$-category with all colimits, and let $B \in \Mack(\mc{G}_X, \mb{C})$ be a Mackey functor, which by Theorem \ref{thm:Gadd} is uniquely determined by $B(X)$ regarded as an object of $\mb{C}$ with $\Aut(X)$-action. Then for any $Y \in \mc{M}$,
\[(i_X)_! B(Y) \simeq (B(X) \X \Hom_{\mc{M}}(X, Y))_{h \Aut(X)}.\]
In particular, if $Y$ is a final object,
\[(i_X)_! B(Y) \simeq B(X)_{h \Aut(X)}.\]
\end{cor}
\qed

\begin{dfn}Let $\mc{M}$ be an EOC and $X \in \mc{M}$ an object. The \emph{Taylor coefficient at $X$} functor is the functor $D^X: \Mack(\mc{M}, \mb{C}) \to \Fun(\mc{G}_X, \mb{C})$ given by $i_X^* \circ \Phi^X$.
\end{dfn}

\begin{dfn}
Let $\mc{M}^\sim$ be the maximal subgroupoid of $\mc{M}$ and define the \emph{Taylor sequence} functor 
\[D : \Mack(\mc{M}, \mb{C}) \to \Fun(\mc{M}^\sim, \mb{C})\]
by
\[\bigvee_{[X]} D^X .\]
\end{dfn}

We'll now enforce the hypothesis that $\mb{C}$ is stable for the remainder of the section. This allows us to state the following important proposition, which is a generalization of the ``norm cofibration sequence" from equivariant stable homotopy theory:
\begin{thm} \label{thm1}
Let $\mc{N}$ be a downwardly closed subcategory of an orbital $\iy$-category $\mc{M}$ and $\mc{T}$ its upwardly closed complement. Denote the restriction $ A^{eff}(i_\mc{T})^*$ and the left Kan extension $A^{eff}(i_\mc{T})_!$ respectively by $\Pi^\mc{T}$ and $\Gamma^\mc{T}$. There's a cofiber sequence of functors $\Mack(\mc{M}, \mb{C}) \to \Mack(\mc{M}, \mb{C})$
\[\Gamma^\mc{T} \Pi^\mc{T} \os{\epsilon} \to \Id  \os{\eta} \to \Xi^\mc{N} \Phi^\mc{N},\]
where $\epsilon$ and $\eta$ are the counit and unit of their respective adjunctions.
\end{thm}
Theorem \ref{thm1} has the following important equivalent form, whose theme is that a Mackey functor contains no secret data not detected by its values:
\begin{cor} \label{mackn0}
$\Xi^\mc{N}$ is fully faithful, and its essential image is the category $\Mack_\mc{N}(\mc{M}, \mb{C})$ of Mackey functors on $\mc{M}$ supported on $\mc{N}$.
\end{cor}
\begin{proof}
If $M \in \Mack_\mc{N}(\mc{M}, \mb{C})$, then evidently $\Pi^\mc{T} M$ is zero, and so the cofiber sequence of Theorem \ref{thm1} collapses to an equivalence
\[M \simeq \Xi^\mc{N} \Phi^\mc{N} M,\]
which establishes the essential image of $\Xi^\mc{N}$.

Since $\Xi^\mc{N}$ is visibly conservative, it follows for $N \in \Mack(\mc{N}, \mb{C})$, setting $M = \Xi^\mc{N} N$, that 
\[N \simeq \Phi^\mc{N} \Xi^\mc{N} N,\]
and the full faithfulness of $\Xi^\mc{N}$ follows by adjunction.
\end{proof}
\begin{rem}
Observe that Corollary \ref{mackn0} also implies Theorem \ref{thm1}. Indeed, it follows abstractly that the functor 
\[M \mapsto \cof(\eps_M)\]
is the localization into $\Mack_\mc{N}(\mc{M}, \mb{C})$. On the other hand, assuming that $\Xi^\mc{N}$ is fully faithful, $\Xi^\mc{N} \Phi^\mc{N}$ is the localization into its essential image. By Corollary \ref{mackn0}, these two localizations coincide, yielding \ref{thm1}.
\end{rem}
\begin{dfn}
Let $\mc{D}A(\mc{M})$ denote the nonabelian derived category of $A^{eff}(\mc{M})$: the category of product-preserving functors $A^{eff}(\mc{M}) \to \Top$ \cite[Definition 5.5.8.8]{HTT}. Equivalently, $\mc{D}A(\mc{M})$ is the category of Mackey functors on $\mc{M}$ valued in the category $\CMon$ of commutative monoid spaces (see, for instance, \cite[Remark 2.7]{Gla16}).

Let $\mb{C}$ and $\mb{D}$ be presentable $\iy$-categories. Recall that the category of coproduct-preserving functors from $\mb{C}$ to $\mb{D}$ is denoted $\Fun^\amalg(\mb{C}, \mb{D})$, and that the category of functors from $\mb{C}$ to $\mb{D}$ that preserve all colimits is denoted $\Fun^L(\mb{C}, \mb{D})$. Now for presentable $\mb{C}$, by \cite[Proposition 5.5.8.15]{HTT}, there's an equivalence of categories
\[\Fun^L(\mc{D}A(\mc{M}), \mb{C}) \simeq \Fun^\amalg(A^{eff}(\mc{M}), \mb{C}).\]
If $\mb{C}$ is, in addition, semiadditive, this can be written as an equivalence
\[\Fun^L(\mc{D}A(\mc{M}), \mb{C}) \simeq \Mack(\mc{M}, \mb{C}). \tag*{$(\ast)$}\]
\end{dfn}
\begin{lem} \label{thm1DA}
Theorem \ref{thm1} holds when $\mb{C} = \CMon$. 
\end{lem}
Before proving Lemma \ref{thm1DA}, let's deduce Theorem \ref{thm1} from Lemma \ref{thm1DA}. We know
\[\Phi^\mc{N} : \mc{D}A(\mc{M}) \to \mc{D}A(\mc{N}) \]
is a localization. Under the equivalence of $(\ast)$, $\Xi^\mc{N}$ corresponds to
\[(\Phi^\mc{N})^* : \Fun^L(\mc{D}A(\mc{N}), \mb{C}) \to \Fun^L(\mc{D}A(\mc{M}), \mb{C}),\]
and by the universal property of a localization, $(\Phi^\mc{N})^*$ is fully faithful, and objects of its essential image are functors which map $\Phi^\mc{N}$-equivalences to equivalences. 

Now here's where we use the stability of $\mb{C}$: since a morphism in a stable $\iy$-category is an equivalence if and only if its fiber is zero, the objects of the essential image of $(\Phi^\mc{N})^*$ are equivalently those functors which map objects in the essential image of 
\[\Gamma^\mc{T} :  \mc{D}A(\mc{T}) \to \mc{D}A(\mc{M})\]
 to $0 \in \mb{C}$. But since the equivalence $(\ast)$ is given by the Yoneda embedding, and since the diagram
\[\begin{tikzcd}
A^{eff}(\mc{T}) \ar{d}[left]{A^{eff}(i_{\mc{T}})} \ar{r}{\sim} & \mc{D}A(\mc{T}) \ar{d}{\Gamma^\mc{T}} \\
A^{eff}(\mc{M}) \ar{r}{\sim} & \mc{D}A(\mc{M}) \end{tikzcd}
\]
commutes, these correspond under $(\ast)$ to the objects of $\Mack_{\mc{N}}(\mc{M}, \mb{C}).$ This proves Corollary \ref{mackn0}, and therefore Theorem \ref{thm1}.
\begin{proof}[Proof of Lemma \ref{thm1DA}.]
The proof of this lynchpin lemma is very technical, and so we've relegated it to Appendix \ref{app:lem}. It's not required reading for those who don't care to learn to make a very specific kind of sausage, but we note that it's our main point of contact with the combinatorics of the effective Burnside category.
\end{proof}

Let $\mc{M}$ be a EOC. For each $X \in \mc{M}$, we may define functors
\[R^X = \Xi^X \circ (i_X)_*\] 
and
\[L^X = \Xi^X \circ (i_X)_!,\]
where $(i_X)_!$ and $(i_X)_*$ are left and right Kan extension respectively. Observe that $(i_X)_!$ and $(i_X)_*$ are fully faithful, since they're Kan extensions along a fully faithful functor. Since we've already seen that $\Xi^X$ is fully faithful, we conclude that both $L^X$ and $R^X$ are fully faithful. Moreover, $R^X$ is right adjoint to the Taylor coefficient functor $D^X$.

Similarly, we can define
\[R = \bigvee_{[X]} R^X\]
and
\[L = \bigvee_{[X]} L^X\]
and $R$ is right adjoint to the Taylor sequence functor $D$.
\begin{prop} \label{prop:sectLD}
$L$ is a section of $D$; that is, $D \circ L$ is homotopic to the identity.
\end{prop}
\begin{proof}
We'll use induction on the number of isomorphism classes of objects in $\mc{M}$. If $\mc{M}$ is a groupoid, both $L$ and $D$ are already the identity. In general, let $X$ be a maximal object of $\mc{M}$. It's clear that $D \circ L$ is homotopic to the identity when restricted to $\mc{G}_X$. 

Then the morphism
\[(i_X)_! (i_X)^* L \to L\]
is equivalent to the summand inclusion
\[L^X \to \bigvee_{[Y]} L^Y\]
and therefore by Theorem \ref{thm1}, we have a cofiber sequence
\[L^X \to L \to \Phi^{\mc{M}_{< X}} L\]
which shows that the left square in the diagram
\[\begin{tikzcd}
\Fun(\mc{M}^\sim, \mb{C}) \ar{r}{L} \ar{d} & \Mack(\mc{M}, \mb{C}) \ar{d}{\Phi^{\mc{M}_{<X}}} \ar{r}{D} & \Fun(\mc{M}^\sim, \mb{C}) \ar{d} \\
\Fun(\mc{M}_{<X}^\sim, \mb{C}) \ar{r}{L} & \Mack(\mc{M}_{<X}, \mb{C}) \ar{r}{D} & \Fun(\mc{M}^\sim_{<X}, \mb{C}  )
\end{tikzcd}\]
commutes. The right square commutes by construction, and the bottom composite is homotopic to the identity by the induction hypothesis. By circumnavigating the diagram, we conclude that $D \circ L$ is homotopic to the identity when restricted to $\mc{G}_Y$ for any $Y \neq X$.
\end{proof}
The following is our Arone-Ching theorem in its general form.
\begin{thm}\label{thm:ac}
Let $\mc{M}$ be an epiorbital category. Then the adjunction $(D, R)$ is comonadic.
\end{thm}
\begin{proof}
$(D, R)$ is comonadic if and only if the natural transformation
\[t : \Id \to \mb{Tot}(\text{Cobar}(R, DR, D))\]
is an equvalence. We'll closely follow Arone and Ching's proof in \cite{AC15}. This involves showing, for each downwardly-closed subcategory $\mc{N}$ of $\mc{M}$, that the natural map 
\[t^{\mc{N}}: \Xi^{\mc{N}} \Phi^{\mc{N}} \to  \mb{Tot}(\Xi^{\mc{N}} \Phi^{\mc{N}} \text{Cobar}(R, DR, D))\]
is an equivalence, by induction on the number of isomorphism classes of objects in $\mc{N}$. Since $\Xi^{\mc{M}} \Phi^{\mc{M}}$ is the identity functor, this will give the result.

So assume that $t^{\mc{P}}$ is an equivalence for all $\mc{P}$ with at most $k$ isomorphism classes of objects, and suppose $\mc{N}$ has $k + 1$ isomorphism classes of objects. Let $X$ be a maximal object of $\mc{N}$ and let $\mc{N}' = \mc{N} \setminus \mc{G}_X$ be the result of removing the isomorphism class of $X$. The cofiber sequence of Theorem \ref{thm1} gives a cofiber sequence of functors
\[L^X D^X \to \Xi^\mc{N} \Phi^\mc{N} \to \Xi^{\mc{N}'} \Phi^{\mc{N}'},\]
which in turn gives a map of cofiber sequences
\[\begin{tikzcd}
L^X D^X \ar{r}{t^X} \ar{d} & \mb{Tot}(L^X D^X \text{Cobar}(R, DR, D)) \ar{d} \\
\Xi^\mc{N} \Phi^\mc{N} \ar{r}{t^\mc{N}} \ar{d} & \mb{Tot}(\Xi^{\mc{N}} \Phi^{\mc{N}} \text{Cobar}(R, DR, D)) \ar{d} \\
\Xi^{\mc{N}'} \Phi^{\mc{N}'} \ar{r}{t^{\mc{N}'}} &  \mb{Tot}(\Xi^{\mc{N}'} \Phi^{\mc{N}'} \text{Cobar}(R, DR, D)).
\end{tikzcd}\]
By the induction hypothesis, $t^{\mc{N}'}$ is an equivalence, so it'll suffice to show that $t^X$ is an equivalence. This also starts the induction, since $t^{\mc{N}} = t^X$ if $\mc{N} = \mc{G}_X$ is a connected groupoid. But now we observe that
\[D^X = ev_X D\]
and so
\begin{align*} \mb{Tot}(L^X D^X \text{Cobar}(R, DR, D)) & \simeq \mb{Tot}(L^X ev_X D \text{Cobar}(R, DR, D)) \\
& \simeq \mb{Tot}(L^X ev_X \text{Cobar}(DR, DR, D)) \\
& \simeq L^X ev_X D \\
& \simeq L^X D^X \end{align*}
by the usual extra codegeneracy argument. This completes the proof.
\end{proof}
The next section aims to characterize the comonad $DR$.

\section{Categories stratified along a poset}\label{sec:strat}
The categories $\Mack(\mc{M})$ for $\mc{M}$ a EOC, along with many other categories occuring in nature, share a significant structural property: any object of $\Mack(\mc{M})$ can be torn open by a series of fracture squares. More precisely, suppose that $\mc{N}$ is a downwardly-closed subcategory, $\mc{T}$ is its upwardly-closed complement and $X \in \Mack(\mc{M})$. Then we'll see that there's a pullback square
\[\begin{tikzcd}
X \ar{r} \ar{d} & (i_{\mc{T}})_* (i_{\mc{T}})^* X \ar{d} \\
\Phi^{\mc{N}} X \ar{r} & \Phi^{\mc{N}} (i_{\mc{T}})_* (i_{\mc{T}})^* X
\end{tikzcd}\]
To build a theory of how $X$ might be recovered from such data, it'll be helpful to widen our scope. First it's important to advertise a potential point of significant notational confusion.
\begin{war}
When we regard a poset as a category in this paper, we will use the \emph{opposite} of the usual convention that there is a morphism from $x$ to $y$ if $x \leq y$. For us, the space of morphisms from $x$ to $y$ will be contractible if $x \geq y$ and empty otherwise. We adopt this strange convention in order to preserve intuitions about size of objects in our chief motivating examples of posets: the posets of isomorphism classes of objects of the EOCs $\mc{O}_G$ and $\F_s^{\leq n}$.
\end{war}
\begin{dfn}
Let $\mc{P}$ be a poset. An \emph{interval} in $\mc{P}$ is a subset $I \inc \mc{P}$ such that whenever $x, y \in I$ and $x < z < y$, we have $z \in I$. If $\mc{P}$ is any poset, then we denote by $\mc{I}_\mc{P}$ be the poset of intervals in $\mc{P}$ ordered by inclusion. If $I$ and $J$ are a pair of intervals, we'll write $I \prec J$ if $I \cap J = \emptyset$ and there is no pair $(i \in I, j \in J)$ with $i > j$.
\end{dfn}
Note that the relation $\prec$ is \emph{not} a partial order: for example, if $p, q \in \mc{P}$ are incomparable, then both $\{p\} \prec \{q\}$ and $\{q\} \prec \{p\}$.
\begin{dfn} \label{dfn:prestrat}
Suppose $\mb{C}$ is a stable $\iy$-category. Let $\mc{E}_\mb{C}$ be the poset of stable reflective subcategories of $\mb{C}$, ordered by inclusion; equivalently, $\mc{E}_\mb{C}$ is the opposite of the poset of exact localizations of $\mb{C}$. Let $\mb{P}$ be a poset. Then a \emph{pre-stratification of $\mb{C}$ along $\mc{P}$} is a map of posets
\[ \mf{S}: \mc{I}_\mc{P} \to \mc{E}_\mb{C}.\]
\end{dfn}
Before we give the criteria that will qualify a pre-stratification as a stratification, it'll be useful to record an elementary fact about localizations.
\begin{lem}\label{lem:recolltfae}
	Given two localizations $\mc{L}_1$ and $\mc{L}_2$ on a stable $\iy$-category $\mb{C}$ such that $\mc{L}_1 \mc{L}_2 = 0$, the following conditions are equivalent:
	\begin{enumerate}[(1)] \item The natural diagram
		\[ \begin{tikzcd} \id \ar{r} \ar{d} & \mc{L}_1 \ar{d} \\ \mc{L}_2 \ar{r} & \mc{L}_2 \mc{L}_1 \end{tikzcd} \]
		is a pullback square.
		\item The containment $\mc{L}_2 \mb{C} \inc \ker(\mc{L}_1)$ is an equality.
		\item $\mc{L}_1$ and $\mc{L}_2$ are jointly conservative.
	\end{enumerate}
\end{lem}
\begin{proof}
	The implication $(1) \Rightarrow (3)$ is obvious. We'll prove $(3) \Rightarrow (2) \Rightarrow (1)$. Suppose that $\mc{L}_1$ and $\mc{L}_2$ are jointly conservative; then if some \[X \in \ker(\mc{L}_1) \setminus \mc{L}_2 \mb{C},\] the localization map $X \to \mc{L}_2 X$ is a non-equivalence which becomes an equivalence after applying either $\mc{L}_1$ or $\mc{L}_2$, establishing $(2)$. If we now denote by $\mc{C}$ the fiber of $\id \to \mc{L}_1$, then 
	\[ \im \mc{C} \inc \mc{L}_2 \mb{C} \] 
	(in fact, $\mc{C}$ is the coreflection into $\mc{L}_2 \mb{C}$). Now taking horizontal fibers in the square diagram gives the morphism
	\[ \mc{C} \to \mc{L}_2\mc{C} \]
	which is an equivalence, establishing $(1)$.
\end{proof}
\begin{dfn}\label{def:strat}
In the notation of Definition \ref{dfn:prestrat}, let
\[\mc{L}_I : \mb{C} \to \mb{C}\] 
be the localization functor corresponding to $\mf{S}(I)$. Then we call $\mf{S}$ a \emph{stratification of $\mb{C}$ along $\mc{P}$} if the following conditions hold:\begin{enumerate}[(1)]
\item $\mf{S}(\mc{P}) = \mb{C}$,
\item if $I_2 \prec I_1$, then $\mc{L}_{I_1} \mc{L}_{I_2} = 0$, 
\item and if $I = I_1 \coprod I_2$, then $\mc{L}_{I_1}$ and $\mc{L}_{I_2}$, viewed as localization functors on $\mf{S}(I)$, satisfy the equivalent conditions of Lemma \ref{lem:recolltfae}.   \end{enumerate}
\end{dfn}
\begin{rem}
Since $\emptyset \prec \emptyset$, axiom (2) implies that $\mc{S}(\emptyset) = \{0\}$.
\end{rem}
\begin{rem}
It follows from the latter two axioms that if $I = I_1 \coprod I_2$ and $I_2 \prec I_1$, then $\mf{S}(I)$ is a \emph{recollement} of $\mf{S}(I_2)$ and $\mf{S}(I_1)$ in the sense of \cite[Definition A.8.1]{HA}.
\end{rem}
\begin{rem}
For most of this section we'll assume that $\mc{P}$ is finite, but our main result generalizes easily to certain infinite posets (Definition \ref{dfn:prostrat}). $n$ will usually denote the cardinality of $\mc{P}$. We'll also assume that $\mc{P}$ is connected; it's easy to see that a category stratified along a disconnected poset decomposes naturally as a direct sum of categories stratified along the connected components.
\end{rem}
Next we'll see some examples of stratifications. Suppose $\mb{C}$ is a presentable symmetric monoidal stable $\iy$-category whose tensor product preserves colimits in each variable, so that we can talk about the homological localization with respect to an object $E \in \mb{C}$ \cite{Bou79}:
\begin{rec}
An object $F \in \mb{C}$ is called \emph{$E$-acyclic} if $E \ot F$ is zero. An object $G \in \mb{C}$ is called \emph{$E$-local} if $\Map(F, G)$ is contractible for any $E$-acyclic $F$. The $E$-local objects of $\mb{C}$ form a reflective subcategory $\mb{L}_E$, with associated localization functor $\mc{L}_E$.
\end{rec}
Let $\mc{P}$ be a poset and suppose we have an object $K_p$ for each $p \in \mc{P}$ such that any $K_p$-local object is $K_q$-acyclic unless $p \geq q$. We can define a pre-stratification $\mf{S}_{K_\bu}$  of $\mb{C}$ along $\mc{P}$ by assigning to $I$ the category of objects which are local with respect to the object $\bigvee_{i \in I} K_i.$
\begin{prop}\label{objstrat}
$\mf{S}_{K_\bu}$ is a stratification of $\mf{S}_{K_\bu}(\mc{P}) = \mb{L}_{\bv_{p \in \mc{P}} K_p}$.
\end{prop}
\begin{proof}
The proof will consist of two lemmas, and will use induction on the cardinality of $\mc{P}$.
\begin{lem}\label{lem:locpull} Suppose $E$ and $F$ are such that any $E$-local object is $F$-acyclic. Then the square of functors
\[\begin{tikzcd}
\mc{L}_{E \vee F} \ar{r}{f} \ar{d}{g} & \mc{L}_F \ar{d}{h} \\
\mc{L}_E \ar{r}{i} & \mc{L}_E \mc{L}_F
\end{tikzcd}\]
is a pullback square. \end{lem}
\begin{proof}This fact is folklore, and cases of it go back to Bousfield and further. The proof, which we now give, is simple.

$g$ and $h$ are both $E$-localizations, and therefore $E$-equivalences, and so the total fiber of the square is $E$-acyclic. On the other hand, $f$ is an $F$-equivalence, and so is $i$ because its source and target are both $F$-acyclic. Thus the total fiber of the square is $F$-acyclic, and so $E \vee F$-acyclic. But everything in the square is $E \vee F$-local, so the total fiber must also be $E \vee F$-local, and hence zero.
\end{proof}
\begin{lem}
Suppose $\mc{I} \inc \mc{P}$ is a proper interval and $p \in \mc{P}$ is such that $\mc{I} \prec \{p\}$. Then any object $E \in \mf{S}_{K_\bu}(\mc{I})$ is $K_p$-acyclic.
\end{lem}
\begin{proof}
The proof will use Theorem \ref{strat} (spoilers). By the induction hypothesis, $\mf{S}_{K_\bu}$ restricts to a stratification of $\mf{S}_{K_\bu}(\mc{I})$ along $\mc{I}$. Then the proof of Theorem \ref{strat} expresses $E$ as a finite limit of objects which are $K_i$-local for some $i \in \mc{I}$, and thus $K_p$-acyclic. Therefore $E$ is $K_p$-acyclic.
\end{proof}
\end{proof}
\begin{rem}
	For this class of stratifications, in the case where $\mc{P}$ is totally ordered, a theorem similar to Theorem \ref{strat} has appeared previously in \cite{A-CB14}.
\end{rem}
We'll now give a pair of quick applications of Proposition \ref{objstrat}.
\begin{exa}
Fix a prime $p$. Then the Morava $K$-theory spectra $K(0), \cdots, K(n)$ give rise to a stratification of the category of $\bv_{i = 0}^n K(i)$-local spectra along $(D^n)^\op$. (Recall that a spectrum is $\bv_{i = 0}^n K(i)$-local if and only if it's local with respect to the Morava $E$-theory $E_n$.) The pullback squares in this stratification include the famous \emph{chromatic fracture squares}
\[\begin{tikzcd}
\mc{L}_{E_n} \ar{r} \ar{d} & \mc{L}_{K_n} \ar{d} \\
\mc{L}_{E_{n-1}} \ar{r} & \mc{L}_{E_{n-1}} \mc{L}_{K_n}
\end{tikzcd}\]
which are the subject of Hopkins' chromatic splitting conjecture \cite{Hov93}.
\end{exa}
\begin{exa}
Let $X$ be a scheme; let $(U_i)_{0 \leq i \leq n}$ be locally closed subschemes of $X$ such $U_0 = X$, $U_1$ is an open subscheme of $X$, and for $i \geq 2$, $U_i$ is an open subscheme of $U_{i - 1} \setminus U_{i - 2}$. We say that the $U_i$ form a \emph{stratification} of $X$.

Let $\mb{QC}(X)$ be the stable $\iy$-category of quasicoherent complexes on $\mb{X}$. Then the structure sheaves $\mc{O}_{U_i} \in \mb{QC}(X)$, $1 \leq i \leq n$, satisfy the hypotheses of Proposition \ref{objstrat} and so form a stratification of $\mc{L}_{\bigoplus_{i=1}^n \mc{O}_{U_i}} \mb{QC}(X)$ along $\D^{n-1}$. In the case where
\[\bigcup_{i =1}^n U_i = X,\]
this is a stratification of $\mb{QC}(X)$ itself.
\end{exa}
Our other main source of examples of stratifications comes from the theory developed in Section \ref{sec:orbmack}. We'll be able to say something about orbital $\iy$-categories and substantially more about epiorbital categories.

Let $\mc{M}$ be an orbital $\iy$-category, let $\mc{N}$ be a downwardly-closed subcategory of $\mc{M}$, and let $\mc{T}$ be its upwardly-closed complement. If $\mb{C}$ is a stable $\iy$-category with all limits and colimits, let $\Mack(\mc{M}; \mb{C})$ be the category of $\mb{C}$-valued Mackey functors on $\mc{M}$ (Definition \ref{mackfun}). We define a pre-stratification $\mf{S}_\mc{M}$ of $\Mack(\mc{M}, \mb{C})$ along $\D^1$ as follows: \begin{itemize}
\item $\mf{S}_\mc{M}(\D^1) = \Mack(\mc{M}, \mb{C})$,
\item $\mf{S}_\mc{M}(\{1\})$ is the category $\Mack^{\mc{T}}(\mc{M}, \mb{C})$ of Mackey functors in the essential image of the right Kan extension from $\Mack(\mc{T}, \mb{C})$,
\item $\mf{S}_\mc{M}(\{0\})$ is the category $\Mack_{\mc{N}}(\mc{M}, \mb{C})$ of Mackey functors supported on $\mc{N}$ (see Corollary \ref{mackn0}).
\end{itemize}
\begin{prop} \label{prop:orbstratD1}
$\mf{S}_\mc{M}$ is a stratification.
\end{prop}
\begin{proof}
We must show that the square
\[\begin{tikzcd} \id_{\Mack(\mc{M}, \mb{C})} \ar{r} \ar{d} & \mc{L}_1 \ar{d} \\
\mc{L}_0 \ar{r} & \mc{L}_0 \mc{L}_1 \end{tikzcd} \]
is a pullback square of endofunctors. But by taking vertical fibers and applying Theorem \ref{thm1}, we're reduced to showing that the natural map
\[\Gamma^\mc{T} \Pi^\mc{T} \to \Gamma^\mc{T} \Pi^\mc{T} \mc{L}_1\]
is an equivalence, which is obvious. 
\end{proof}

Now suppose $\mc{M}$ is epiorbital, and let $\mc{P}_\mc{M}$ be the poset of isomorphism classes in $\mc{M}$ (Definition \ref{orbposet}).
\begin{dfn}
We define a pre-stratification $\mf{S}_\mc{M}$ on $\Mack(\mc{M}; \mb{C})$ as follows. If $I \inc \mc{P}_\mc{M}$ is any interval, let $\mc{I}$ be the corresponding full subcategory of $\mc{M}$. If $I$ is downwardly-closed, then we define
\[\mf{S}_\mc{M}(I) = \Mack_{\mc{I}}(\mc{M}; \mb{C}) \tag*{(Corollary \ref{mackn0})}\]
If $J \subseteq \mc{P}_\mc{M}$ is upwardly closed, then let $\Mack^{\mc{J}}(\mc{M}, \mb{C})$ be the essential image of the right Kan extension
\[A^{eff}(i_{\mc{J}})_* : \Mack(\mc{J}, \mb{C}) \to \Mack(\mc{M}, \mb{C}).\]
We define
\[\mf{S}_\mc{M}(J) = \Mack^{\mc{J}}(\mc{M}, \mb{C}).\]
If $I \inc \mc{P}_\mc{M}$ is any interval, then we can write
\[I = I_+ \cap I_-\]
where $I_+$ is the smallest upwardly-closed set containing $I$ and $I_-$ is, likewise, the smallest downwardly-closed set containing $I$. Then we define
\[\mf{S}_\mc{M}(I) = \Mack^{\mc{I}_+}(\mc{M}, \mb{C}) \cap \Mack_{\mc{I}_-}(\mc{M}, \mb{C}).\]
\end{dfn}
\begin{prop}
$\mf{S}_\mc{M}$ is a stratification of $\Mack(\mc{M}, \mb{C})$ along $\mc{P}_\mc{M}$.
\end{prop}
\begin{proof}
Clearly 
\[\mf{S}_\mc{M}(\mc{P}_\mc{M}) = \Mack(\mc{M}, \mb{C}).\]
We must verify $(3)$ in Definition \ref{def:strat}. Let $\mc{I}, \mc{I}_1, \mc{I}_2$, be the full subcategories of $\mc{M}$ corresponding respectively to $I, I_1, I_2$, and let $\mc{D}$ be the smallest downwardly-closed subcategory of $\mc{M}$ containing $\mc{I}$. Then by passing to $\Mack_\mc{D}(\mc{M}, \mb{C})$ if necessarily, we may assume that $I_1$ is upwardly-closed.

Assume for a moment that $I = \mc{P}_\mc{M}$. Then, as in the proof of Proposition \ref{prop:orbstratD1}, we conclude by taking fibers vertically and invoking Theorem \ref{thm1}. In general, we note that
\[I_2 = I \cap (\mc{P}_\mc{M} \setminus I_1).\]
We know that the natural diagram
\[\begin{tikzcd}
\id \ar{r} \ar{d} & \mc{L}_{I_1} \ar{d} \\
\mc{L}_{\mc{P}_\mc{M} \setminus I_1} \ar{r} & \mc{L}_{\mc{P}_\mc{M} \setminus I_1} \mc{L}_{I_1}
\end{tikzcd}\]
is a pullback square, and applying $\mc{L}_I$ to the entire square gives the result.
\end{proof}

We'll now start setting up for our main result on the reconstruction of objects in a stratified category from their atomic localizations.
\begin{dfn}
Let $\mc{P}$ be a finite poset. Define $\Su(\mc{P})$ to be the poset of nonempty subsets $T = \{i_1, i_2, \cdots, i_k\}$ of $\mc{P}$, ordered by reverse inclusion.
\end{dfn}
\begin{dfn}
Let $(\mb{C}, \mf{S})$ be a stable $\iy$-category stratified along $\mc{P}$. We define an $\iy$-category $\mb{C}^\mf{S}$ as the full subcategory of $\Fun(\Su(\mc{P}), \mb{C})$ spanned by those functors
\[F : \Su(\mc{P}) \to \mb{C}\]
 such that
\begin{itemize}
\item for each $T \in \Su(\mc{P})$ and for each minimal element $t \in T$, $F(T)$ is in $\mf{S}(\{t\})$;
\item if $e$ is an edge of $\Su(\mc{P})$ of the form $T \to T \cup \{p\}$ where $\{p\} \prec T$, then $F(e)$ exhibits $F(T \cup \{p\})$ as the $\mf{S}(\{p\})$-localization of $F(T)$.
\end{itemize}
\end{dfn}
The following, describing how objects of a stratified category can be assembled using higher fracture squares, is the main theorem of this section. 
\begin{thm} \label{strat}
There's an equivalence of categories
\[\mf{d} : \mb{C} \to \mb{C}^\mf{S}.\]
This equivalence will be constructed as an explicit zigzag in the course of the proof.
\end{thm}

The first step is to realize that we don't have enough posets, and define some more posets.
\begin{dfn}
Let $\Su'(\mc{P})$ be the set whose elements are nonempty sets $\{I_1, \cdots, I_k\}$ of nonempty, disjoint intervals in $\mc{P}$ such that for each pair of indices $i, j$, either $I_i \prec I_j$ or $I_j \prec I_i$. We'll put a partial order on $\Su'(\mc{P})$ by letting 
\[\{I_1, \cdots, I_k\} \geq \{J_1, \cdots, J_l\}\] 
if there exist distinct indices $(i_1, \cdots, i_k)$ such that $J_{i_j} \inc I_j$ for each $j$.
\end{dfn}
If $T \in \Su'(\mc{P})$ and $I \in T$, then we'll call $I$ \emph{$\prec$-minimal} if for each $J \in T$ with $J \neq I$, $I \prec J$. This doesn't necessarily imply that there is no $J \in T$ such that $J \prec I$. However, since $\prec$ is transitive, each $T \in \Su'(\mc{P})$ has at least one $\prec$-minimal element.

Here's a way of ``coordinatizing" a poset.
\begin{dfn}
A \emph{threading} of a finite poset $\mc{P}$ is a filtration
\[\emptyset = \mc{P}_{\leq 0} \inc \mc{P}_{\leq 1} \inc \cdots \mc{P}_{\leq n} = \mc{P}\]
such that for each $i$, $\mc{P}_{\leq i}$ is downwardly closed, and 
\[|\mc{P}_{\leq i}| = i.\]
\end{dfn}
We'll fix, once and for all, a threading on $\mc{P}$.
\begin{dfn}
Let $\Su^m(\mc{P})$ be the subset of $\Su'(\mc{P})$ containing those sets \[T = \{I_1, \cdots, I_k\}\] such that for each $I_i$, either $|I_i| = 1$ or 
\[I_ i = \mc{P}_{\leq i} \text{ for some } i \leq m.\]
 In particular, at most one of the $I_i$ may have cardinality $>1$, and if this occurs then $I_i$ must be a $\prec$-minimal element of $T$.
\end{dfn}
There's an obvious isomorphism $\Su^1(\mc{P}) \cong \Su(\mc{P})$.
\begin{dfn}
Let $\mb{C}^\mf{S}_m$ be the full subcategory of $\Fun(\Su^m(\mc{P}))$ spanned by those functors
\[F: \Su^m(\mc{P}) \to \mb{C}\]
such that
\begin{itemize}
\item for each $T = \{I_1, \cdots, I_k\} \in \Su^m(\mc{P})$ and for each $I_i$ that is $\prec$-minimal in $T$, $F(T) \in \mf{S}(I_i)$;
\item if $e : T_1 \to T_2$ is an edge of $\Su^m(\mc{P})$ of the form 
\[\{I_1, I_2, \cdots, I_k\} \to \{I_1, I_2, \cdots, J_k\}\]
 with $I_k$ $\prec$-minimal in $T_1$ and $J_k \inc I_k$, then $F(e)$ exhibits $F(T_2)$ as the $\mf{S}(J_k)$-localization of $F(T_1)$;
\item if $e : T_1 \to T_2$ is an edge of $\Su^m(\mc{P})$ of the form 
\[\{I_1, I_2, \cdots, I_k\} \to \{I_1, I_2, \cdots, I_k, I_{k + 1}\}\]
with $I_{k + 1}$ $\prec$-minimal, then $F(e)$ exhibits $F(T_2)$ as the $\mf{S}(I_{k+1})$-localization of $F(T_1)$.
\end{itemize}
\end{dfn}
\begin{prop} \label{strat1}
If $\mc{P}$ has cardinality at most $m$, then the functor
\[e_\mc{P} : \mb{C}^\mf{S}_m \to \mb{C}\]
given by evaluation at $\{\mc{P}\}$ is an equivalence of categories.
\end{prop} 
\begin{proof}
We'll prove this by induction on $m$. If $m = 1$, there's nothing to do. If $m = 2$, then $\mc{P} \cong \Delta^1$ and we'll label its elements $0$ and $1$, with $0 > 1$ (this is unfortunately necessitated by our conventions). Then $\mb{C}^\mf{S}_m$ is the category of squares of the form
\[\begin{tikzcd}
E \ar[tail]{r} \ar[tail]{d} & E_0 \ar[tail]{d} \\
E_1 \ar{r} & E_{10}
\end{tikzcd}\]
in which $E_0 \in \mf{S}(0)$, $E_1, E_{10} \in \mf{S}(1)$ and the tailed arrows are localizations. By the stratification axioms, all such squares are cartesian. Then the fact that $e_{\D^1} : \mb{C}^\mf{S}_m \to \mb{C}$ is an equivalence is discussed, in almost exactly these terms, in the proof of \cite[Proposition A.8.11]{HA}.

In general, let $x \in \mc{P}$ be the unique element of $\mc{P} \setminus \mc{P}_{\leq n - 1}$; then $x$ is a maximal element. For convenience, we'll write $\mc{Q}$ for $\mc{P}_{\leq n -1}$. Observe that we have an isomorphism of posets
\[h: \D^1 \X ((\Su^m(\mc{Q}))^\lhd) \to \Su^m(\mc{P})\]
given by
\begin{align*} & h(0, c) = \{\mc{P}\}, \\
&h(1, c) = \{\{x\}\}, \\
&h(0, T) = T, \\
&h(1, T) = \{\{x\}\} \cup T, \end{align*}
where $c$ is the cone point. 

Now $\mb{C}$ admits a stratification $\mf{S}'$ along $\D^1$ wherein
\[\mf{S}'(1) = \mf{S}(\{x\}), \, \, \mf{S}'(0) = \mf{S}(\mc{Q}),\]
so that 
\[\phi_{\D^1} : \mb{C}^{\mf{S}'}_m \to \mb{C}\]
is an equivalence. On the other hand, $\mf{S}(\mc{Q})$ obviously inherits a stratification $\mf{S}''$ along $\mc{Q}$, and the functor 
\[\phi_{\mc{Q}} : \mf{S}(\mc{Q})^{\mf{S}''}_m \to \mf{S}(\mc{Q})\]
is an equivalence.

Here's what we deduce by combining these two equivalences. Let 
\[\mc{K} = \Su^m(\D^1) \amalg_{\D^1} (\Su^m(\mc{Q}) \X \D^1)\]
where we've glued the edge $\{0\} \to \{0, 1\}$ of $\Su^m(\D^1)$ to the edge $\{\mc{Q}\} \X \D^1$ of $\Su^m(\mc{Q}) \X \D^1$. Let $\ol{\mb{C}}$ be the full subcategory $\Fun(\mc{K}, \mb{C})$ spanned by those $F$ for which
\[F|_{\Su^m(\D^1)} \in \mb{C}^{\mf{S}'}_m\]
and
\[F|_{\Su^m(\mc{Q}) \X \{0\}}, F|_{\Su^m(\mc{Q}) \X \{1\}} \in \mf{S}(\mc{Q})^{\mf{S}''}_m.\]
Then evaluation on $\{\D^1\} \in \Su^m(\D^1)$ induces an equivalence of categories 
\[e_{\D^1} :\ol{\mb{C}} \toe \mb{C}.\]
But 
\[\Su^m(\D^1) \cong \D^1 \X \D^1,\]
and so
\[\mc{K} \cong (\D^1 \amalg_{\{1\}} \Su^m(\mc{Q})) \X \D^1.\]
For any simplicial set $S$, the inclusion
\[\D^1 \amalg_{\{1\}} S \inj S^\lhd\]
is inner anodyne, and so we get an inner anodyne composite
\[\mc{K} \inj \D^1 \X ((\Su^m(\mc{Q}))^\lhd) \os{h} \to \Su^m(\mc{P}),\]
giving an equivalence of categories $\Fun(\Su^m(\mc{P}), \mb{C}) \simeq \Fun(\mc{K}, \mb{C})$, which restricts to an equivalence of categories $\mb{C}^\mf{S}_m \to \ol{\mb{C}}$. Composing with $e_{\D^1}$ completes the proof.
\end{proof}
\begin{prop} \label{strat2}
For any $\mc{P}$ and any $\mb{C}$ stratified along $\mc{P}$, the restriction functor
\[r_m : \mb{C}^\mf{S}_m \to \mb{C}^\mf{S}_{m - 1}\]
is an equivalence.
\end{prop}
\begin{proof}
Let 
\[\kappa_m : \mb{C}^\mf{S}_{m - 1} \to \Fun(\Su^m(\mc{P}), \mb{C})\]
be the right Kan extension functor. We claim that $\mb{C}^\mf{S}_m$ is equal to the essential image of $\kappa_m$.

Indeed, let $q$ be the unique element of $\mc{P}_{\leq m} \setminus \mc{P}_{\leq m -1}$. Suppose 
\[T \in \Su^m(\mc{P}) \setminus \Su^{m-1}(\mc{P}).\]
Then $T$ is of the form
\[\{ \{p_1\}, \cdots, \{p_k\}, \mc{P}^{\leq m} \}.\]
Let $\alpha: \La^2_2 \to \Su^{m-1}(\mc{P})$ be the functor with
\begin{align*} & \alpha_T(0) = \{ \{p_1\}, \cdots, \{p_k\}, \{q\}\}, \\
& \alpha_T(1) = \{ \{p_1\}, \cdots, \{p_k\}, \mc{P}_{\leq m -1}\}, \\
& \alpha_T(2) =  \{ \{p_1\}, \cdots, \{p_k\}, \{q\}, \mc{P}_{\leq m-1}\}.\end{align*}
Then $\alpha$ is coinitial in $\Su^{m-1}(\mc{P})_{T/}$. Moreover, $F : \Su^m(\mc{P}) \to \mb{C}$ is an object of $\mb{C}^\mf{S}_m$ if and only if
\begin{itemize}
\item $F|_{\Su^{m-1}(\mc{P})} \in \mb{C}^\mf{S}_{m - 1}$, and
\item for each $T \in \Su^m(\mc{P}) \setminus \Su^{m-1}(\mc{P})$, $F(T) \in \mf{S}(\mc{P}_{\leq m})$ and the maps 
\[F(T) \to F(\alpha_T(0)), F(T) \to \F(\alpha_T(1))\]
are localizations.
\end{itemize}
But by the stratification axiom, the latter condition is equivalent to the condition that the square
\[\begin{tikzcd}
F(T) \ar{r} \ar{d} & F(\alpha_T(0)) \ar{d} \\
F(\alpha_T(1)) \ar{r} & F(\alpha_T(2)) \\
\end{tikzcd}\]
be a pullback. This completes the proof.
\end{proof}
If $n= |\mc{P}|$, we now have equivalences of categories $\mb{C}^\mf{S}_n \toe \mb{C}$ (Proposition \ref{strat1}) and $\mb{C}^\mf{S}_n \toe \mb{C}^\mf{S}_1 \cong \mb{C}^\mf{S}$ (inductively, using Proposition \ref{strat2}). This consitutes a proof of Theorem \ref{strat}.

\begin{exa}
When $\mb{C}$ is the category $\Sp^G$ for a finite abelian group $G$, we recover the statement of \cite[Theorem 3]{AK13}, though in substantially different language.
\end{exa}

\begin{exa} \label{exa:Cp}
Suppose $p$ is a prime and $\mc{M} = \mc{O}_{C_p}$, so that $\mc{P} = \D^1$. Then Theorem \ref{strat} states, after unwinding the definition, that an object $E$ of of $\Mack(\mc{M}) \simeq \Sp^{C_p}$ is given by the following data:
\begin{itemize}
\item A spectrum with $C_p$-action 
\[E_1 \in \mf{S}_\mc{M}(\{\{1\}\}) \simeq \Fun(BC_p, \Sp),\]
the underlying spectrum of $E$;
\item a spectrum
\[E_0 \in \mf{S}_\mc{M}(\{\{0\}\}) \simeq \Sp,\]
the $C_p$-geometric fixed point spectrum of $E$;
\item and a map
\[E_0 \to \mc{L}_{\{0\}} E_1 \simeq E_1^{t C_p}\]
where $(-)^{tG}$ is the Tate spectrum, defined by the cofiber sequence
\[(-)_{hG} \to (-)^{hG} \to (-)^{tG}.\]
\end{itemize}
The epiorbital category $\F_s^{\leq 2}$ is visibly equivalent to $\mc{O}_{C_2}$, so an object 
\[F \in \Mack(\F_s^{\leq 2}) \simeq \Fun^{2-exc}(\Sp, \Sp)\]
is given by the same data as an object of $\Sp^{C_2}$, but in this case, $E_1$ and $E_0$ are interpreted as the second and first derivatives of $F$, respectively. This classification of 2-excisive functors was first carried out in \cite[\S 5]{AC15}.
\end{exa}

We'll close this section by saying a few words about what happens for infinite posets. Let $\mc{P}$ be an infinite poset equipped with a system of finite subposets
\[\emptyset = \mc{P}_{\leq 0} \inc \mc{P}_{\leq 1} \inc \cdots \inc \mc{P}_{\leq n} \inc \cdots \inc \mc{P}\]
which is a threading in the sense that for each $i$, $\mc{P}_{\leq i}$ is downwardly closed and has cardinality $i$, and
\[\bigcup_n \mc{P}_{\leq n} = \mc{P}.\]
\begin{dfn}\label{dfn:prostrat}
Suppose
\[0 = \mb{C}_0 \inc \mb{C}_1 \inc \cdots \inc \mb{C}_n \inc \cdots\]
is a sequence of stable $\iy$-categories such that $\mb{C}_{n - 1}$ is a reflective stable subcategory of $\mb{C}_n$ for all $n$. Suppose we have, for each $n$, a stratification $\mf{S}_n$ of $\mb{C}_n$ along $\mc{P}_{\leq n}$, and that all of these are compatible in the sense that
\[\mb{C}_{n - 1} = \mf{S}_n(\mc{P}_{\leq n - 1})\]
and $\mf{S}_{n - 1}$ is the induced stratification.
Let 
\[\mb{C}_\iy := \llim{n} \mb{C}_n\]
be the limit over the localization maps. We call this data a \emph{pro-stratification} of $\mb{C}_\iy$ along $\mc{P}$.
\end{dfn}

Then it follows from the proof of Theorem \ref{strat} that the diagram
\[\begin{tikzcd}
\mb{C}_n \ar{d}[left]{\mb{L}_{\mc{P}_{\leq n - 1}}} & (\mb{C}_n)_n^{\mf{S}_n} \ar{l}[below]{e_{\mc{P}_{\leq n}}}[above]{\sim} \ar{r}[above]{\sim} \ar{d} & \mb{C}_n^{\mf{S}_n} \ar{d} \\
\mb{C}_{n - 1} & (\mb{C}_{n- 1})^{\mf{S}_{n - 1}}_{n-1} \ar{l}[below]{e_{\mc{P}_{\leq n-1}}}[above]{\sim} \ar{r}[above]{\sim} & \mb{C}_{n - 1}^{\mf{S}_{n - 1}}
\end{tikzcd}\]
commutes up to homotopy for every $n$. Taking the limit as $n \to \infty$ gives an equivalence
\[\mb{C}_\iy  \to \llim{n} \mb{C}_n^{\mf{S}_n} =: \mb{C}_\iy^\mf{S}.\]
The limit $\mb{C}_\iy^\mf{S}$ can be described explicitly as follows. Let
\[\bb{P}^\iy = \colim{n} \bb{P}(\mc{P}_n).\]
and
\[\mb{C}^\iy := \colim{n} \mb{C}_n\]
(which differs from $\mb{C}_\iy$ in that we have taken the colimit over the inclusions rather than the limit over the localizations). Then $\mb{C}_\iy^\mf{S}$ is the full subcategory of $\Fun(\bb{P}^\iy, \mb{C}^\iy)$ spanned by those functors $F$ for which
\[F|_{\bb{P}(\mc{P}_n)} \in \mb{C}_n^{\mf{S}_n}.\]
Thus we have a description of $\mb{C}_\iy$ in terms of (infinite) diagrams of maximally local objects.
\begin{exa}
With $\mc{P}$ as above, let $\mb{C}$ be a symmetric monoidal presentable stable $\iy$-category and let $(K_p)_{p \in \mc{P}}$ be a collection of objects of $\mb{C}$ such that any $K_p$-local object is $K_q$-acyclic unless $p \geq q$. Then letting 
\[\mb{C}_n = \mb{L}_{\bigvee_{p \in \mc{P}_{\leq n}} K_p} \mb{C}\]
gives a pro-stratification of 
\[\mb{C}_\iy = \mb{L}_{\bigvee_{p \in \mc{P}} K_p} \mb{C}\]
along $\mc{P}$. In the case where $\mc{P} = \bb{N}^\op$ and $K_n$ is the Morava $K$-theory $K(n)$, we have expressed the category of harmonic spectra in terms of diagrams of $K(n)$-local spectra. \end{exa}
\begin{exa} \label{promack}
Let
\[\emptyset = \mc{M}_0 \inc \mc{M}_1 \inc \cdots \inc \mc{M}_n \inc \cdots \inc \mc{M} = \colim{n} \mc{M}_n\]
be a sequence of inclusions of EOCs giving rise to the threading
\[\mc{P}_0 \inc \mc{P}_1 \inc \cdots \inc \mc{P}_n \inc \cdots\]
on posets of isomorphism classes. Then $\mc{M}^\amalg$ is disjunctive and we may speak of the category $\Mack(\mc{M}, \mb{C})$ of additive functors from $A^{eff}(\mc{M}^\amalg)$ into some stable target category $\mb{C}$. Letting
\[\mb{C}_n = \Mack(\mc{M}_n, \mb{C})\]
gives a pro-stratification of $\Mack(\mc{M}, \mb{C})$ along $\mc{P}$.
\end{exa}
Example \ref{promack} has a couple of interesting special cases:
\begin{exa}
Let $\mc{M}_n = \F_s^{\leq n}$. Then $\mc{M}$ is the category $\F_s$ of all finite sets and surjective maps, and $\Mack(\mc{M}, \mb{C})$ is equivalent to the category of functors $F : \Sp \to \mb{C}$ which are \emph{weakly analytic} in the sense that 
\[F \simeq \llim{n} P_n F.\]
\end{exa}
\begin{exa}
Let $G$ be a profinite group and let $\mc{M}$ be the category of finite $G$-orbits. Then any threading of the poset $\mc{P}$ of isomorphism classes in $\mc{M}$ gives a pro-stratification of $\Mack(\mc{M}, \mb{C})$, which should be thought of a kind of category of genuine $G$-objects in which only the fixed points under cofinite subgroups are salient.

It's worth noting that given a cofinal system of finite quotients $(H_m)_{m \in \bb{N}}$ of $G$, we could choose our threading so that for each $m$, there is some $n_m$ such that
\[\mc{M}_{n_m} = \mc{O}_{H_m}.\]
Thus
\[\Mack(\mc{M}, \mb{C}) \simeq \lim_{G \twoheadrightarrow H, H \text{ finite}} \Mack(\mc{O}_H, \mb{C}).\]
\end{exa}

\section{$K(n)$-local theory}\label{sec:knlo}
In this brief section, we'll see that symmetry properties which emerge when one works locally with respect to the Morava K-theories $K(n)$ cause large chunks of this theory to collapse. We'll reprove a result of Kuhn on the $K(n)$-local splitting of Taylor towers, and give a new tom Dieck-like splitting result for $K(n)$-local $G$-spectra.

The following ``chromatic blueshift" theorem is a consequence of the results of \cite{GS96} and \cite{HS96}; it appears in roughly this form in \cite{HL13}, and as we shall see, \cite{Kuh04a} is also highly relevant.

\begin{thm}
Let $G$ be a finite group and let $E$ be a $K(n)$-local spectrum with $G$-action. Then the transfer map
\[N : E_{hG} \to E^{hG}\]
is a $K(n)$-local equivalence. Thus the Tate spectrum $E^{tG}$ is $K(n)$-acyclic.
\end{thm}

\begin{thm}\label{thm:knlo}
If $\mc{M}$ is a epiorbital category and $\mb{C}$ is a stable $\iy$-category such that all Tate spectra are zero - for instance, the category of $K(n)$-local spectra - then the comonad $DR$ of Theorem \ref{thm:ac} is the identity comonad, and so the Taylor sequence functor 
\[D : \Mack(\mc{M}, \mb{C}) \to \Fun(\mc{M}^\sim, \mb{C})\]
is an equivalence.
\end{thm}
Having got this far, the proof is fairly simple.
\begin{proof}
Let $X$ be an object of $\mc{M}$ and let $i_X$ once again denote the full inclusion $A^{eff}(\mc{G}_X) \inj A^{eff}(\mc{M}_{\leq X})$, where $\mc{G}_X$ is the full subcategory of $\mc{M}$ spanned by $X$. Let $E \in \Fun(\mc{G}_X, \mb{C})$. Then for any $Y \in \mc{M}_{\leq X}$, the natural map
\[(i_X)_!(E)(Y) \to (i_X)_*(E)(Y)\]
takes the form
\[\bigoplus_{f \in \Map_\mc{M}(X, Y) / \text{isomorphism}} E_{h \Aut_f(X)} \to \bigoplus_{f \in \Map_\mc{M}(X, Y) / \text{isomorphism}} E^{h \Aut_f(X)}\]
and is thus an equivalence, by our hypothesis. We deduce that for each $X \in \mc{M}$, 
\[L^X \simeq R^X,\]
and so
\[L \simeq R.\]
But $DL \simeq \id$ (Proposition \ref{prop:sectLD}) and so $DR \simeq \id$. This completes the proof.
\end{proof}
\begin{cor}
Any $K(n)$-local $G$-spectrum $E$ (by which we mean a Mackey functor valued in the $K(n)$-local category) satisfies a very strong tom Dieck splitting property: we have an equivalence
\[E \simeq \bigvee_{H \leq G / \text{conjugacy}} L^{G/H} E^{\Phi H}.\]
In particular, for each $H \leq G$, we have a canonical decomposition
\[E^H \simeq \bigvee_{(K \leq H) / \text{conjugacy in } G} \left(  (E^{\Phi K})_{h W(H, K)} \right)\]
where $W(H, K)$ is the relative Weyl group, defined as
\[W(H, K) := (N_G(K) \cap H)/K.\]
\end{cor}
\begin{cor}[Kuhn]
Let $F : \Sp \to \Sp$ be an $m$-excisive functor taking values in $K(n)$-local spectra. Then the Taylor tower for $F$ splits: we have an equivalence
\[F(X) \simeq \bigvee_{i = 0}^m \bb{D}_i F(X).\]
\end{cor}
\appendix
\section{The free semiadditive $\iy$-category on a group} \label{app:Gadd}

This appendix is devoted to proving Theorem \ref{thm:Gadd}, which we restate here (with slightly different notation) for convenience: 
\begin{thm}
Let $G$ be an (ordinary) groupoid. Then $A^{eff}(G^\amalg)$ is the free semiadditive $\iy$-category on $G$: for any semiadditive $\iy$-category $\mb{C}$, the natural inclusion induces an equivalence of categories
\[\Fun^\oplus(A^{eff}(G^\amalg), \mb{C}) \to \Fun(G, \mb{C}).\]
\end{thm}
First, we note that we may assume $G$ is connected. Indeed, having proved this, the general case will follow from the fact that if $(\mc{M}_i)_{i \in I}$ is an $I$-indexed family of orbital categories, then
\[ A^{eff}\left( \left( \coprod_{I}\mc{M}_i \right)^{\amalg}   \right) \simeq \bigoplus_{I} A^{eff}(\mc{M}_i^{\amalg}).  \]
We will further assume that our connected groupoid has only one object, and denote the corresponding group, too, by $G$.

Now let's get some notation out of the way. Let $\F_*$ be the category of finite pointed sets. If $S \in \F_*$, denote by $S^o$ the finite set $S \setminus \{*\}$. If $s \in S^o$, denote by $\chi_s : S \to \{s\}_+$ the characteristic map at $s$:
\[\chi_s(t) = \begin{cases} s & t = s \\ * & \text{otherwise.} \end{cases}\]
\begin{dfn}{\cite[Remark 2.4.2.2]{HA}}
Let $\mb{C}$ be an $\iy$-category which admits finite products. Recall that by definition, the category $\CMon(\mb{C})$ of \emph{commutative monoids} in $\mb{C}$ is the full subcategory of $\Fun(\F_*, \mb{C})$ spanned by those functors $F$ satisfying the Segal condition: for each $S \in \F_*$, the edges $F(S) \to F(\{s\}_+)$ determine an equivalence
\[F(S) \simeq \prod_{s \in S^o} F(\{s\}_+).\]
We'll abbreviate $\CMon(\Top)$ to $\CMon$.
\end{dfn}

Now let's begin the proof. First we note that $G^\amalg$ is equivalent to the category $\F r^\mc{G}$ of finite sets with \emph{free} $G$-action.

\begin{dfn}
Let $\mc{L}(G)$ be the Lawvere theory of commutative monoids with $G$-action: the full subcategory of $\Fun(G, \CMon)$, which is equivalent to $\CMon(\Fun(G, \Top))$, spanned by the the essential image of $\F \subseteq \Top$ under the left adjoint of the forgetful functor $\Fun(G, \CMon) \to \Top$.
\end{dfn}

\begin{thm}
There is an equivalence of categories between $A^{eff}(\FfG)$ and $\mc{L}(G)$.
\end{thm}

\begin{proof}
First let's construct the functor. $\Fun(G, \CMon)$ is a certain full subcategory of $\Fun(G \X \F_*, \Top)$, so we can do this by constructing a functor
\[A^{eff}(\FfG) \X G \X \F_* \to \Top\]
adjointing over, and checking it makes sense on objects.

We'll do the construction in two stages. First, note that we have a functor 
\[A^{eff}(\FfG) \X A^{eff}(\F) \to A^{eff}(\FfG)\]
simply by taking objectwise products of staircase diagrams. We also have an inclusion $i: \F_* \to A^{eff}(\F)$ as follows: an $n$-simplex of $\F_*$, given by a chain of pointed maps 
\[X_0 \os{f_1} \to X_1 \os{f_2} \to \cdots \os{f_n} \to X_n\]
maps to the staircase diagram $(A_{ij})_{0 \leq i \leq j \leq n}$ with
\[\begin{cases} A_{ii} = X_i^o & i = j \\
A_{ij} = (f_j f_{j - 1} \cdots f_{i + 1})^{-1} X_j^o & i \neq j.
\end{cases} \]
Functoriality is easily checked, as is the fact that all squares which ought to be pullbacks are pullbacks. Composing these two and multiplying by $G$, we get a map
\[\mu: A^{eff}(\FfG) \X G \X \mc{F}_* \to A^{eff}(\FfG) \X G.\]

Next, we'll define a functor $A^{eff}(\FfG) \X G \to \Top$ by defining a left fibration 
\[\kappa: A^{eff}(\FfG)_+ \ltimes G \to A^{eff}(\FfG) \X G.\]
Here $A^{eff}(\FfG)_+ \ltimes G$ is itself the total space of a cocartesian fibration over $G$. A vertex of $A^{eff}(\FfG)_+ \ltimes G$ is a free $G$-set $U$ together with a finite set $S$ and a map of sets $S \to U$. An edge of $A^{eff}(\FfG)_+ \ltimes G$ with source $a_1: S_1 \to U_1$ and target $a_2: S_2 \to U_2$ is an element $g \in G$ together with a diagram
\[\begin{tikzcd}
S_1 \ar{d}{g \circ a_1} & S_2 \ar[white]{dl}[very near start, black]{\urcorner} \ar{l} \ar{d} \ar[bend left=20]{dd}{a_2} \\
U_1 & \ar{l} W \ar{d} \\
& U_2
\end{tikzcd}\]
with higher simplices defined analogously. We define $\kappa$ to be the map that forgets $S$. To prove that $\kappa$ is a left fibration, define $A^{eff}(\FfG)_+$ to be the fiber of $A^{eff}(\FfG)_+ \ltimes G$ over the vertex of $G$. $\kappa$ restricts to a map
\[A^{eff}(\FfG)_+ \to A^{eff}(\FfG)\]
which is actually isomorphic to the target map
\[t: A^{eff}(\FfG)_{G/} \to A^{eff}(\FfG)\]
which is definitely a left fibration. Together with the fact that the preimage under $\kappa$ of an edge of $G$ is an equivalence, this implies that $\kappa$ itself is a left fibration.

Let $K$ be a functor $A^{eff}(\FfG) \X G \to \Top$ that classifies $\kappa$. We now have a well-defined functor 
\[\sigma : A^{eff}(\FfG)  \to \Fun(G \X \F_*, \Top)\]
defined by composing $\mu$ with $K$ and then taking adjoints. For each free finite $G$-set $U$, we must show that the functor
\[\F_* \X \{U\} \os{\mu} \to A^{eff}(\FfG) \to \Fun(G, \Top)\]
is a commutative monoid. Unwinding the definitions shows that this is a consequence of the fact that $K$ preserves products. Moreover, one identifies $\sigma(U)$ with 
\[(\Sigma)^{\X U},\]
where $\Sigma = \coprod_{n \geq 0}  \Sigma_n$ and $G$ acts by permutation on the factors. This is equivalent to the free commutative monoid in $\Fun(G, \Top)$ on the set $U/G$. So $\sigma$ factors through a functor
\[\alpha: A^{eff}(\FfG) \to \mc{L}(G).\]

From here, showing that $\alpha$ is an equivalence is the easy part. Essential surjectivity is obvious. For full faithfulness, it suffices to show that $\alpha$ induces an equivalence
\[\alpha_0: \Map_{A^{eff}(\FfG)}(G, G) \to \Map_{\mc{L}(G)}(\Sigma^{\X G}, \Sigma^{\X G})\]
since all of the other relevant maps are products of some copies of this one.

Since $G$ is a commutative monoid in $A^{eff}(\FfG)$ and $\Sigma^{\X G}$ is a commutative monoid in $\mc{L}(G)$, $\alpha_0$ underlies a map of commutative monoids, both of which are easily seen to be equivalent as commutative monoids to $\Sigma^{\X G}$. Thus it's enough to check that $\alpha_0$ takes a set of free generators to a set of free generators. On the left, we may take this set to be
\[\left( \begin{tikzcd}[column sep = small] & G \ar[equals]{ld} \ar{dr}{r_g} \\ G && G \end{tikzcd} \right)_{g \in G} \]
where $r_g$ is right multiplication by $g$. On the other hand, we may take our set of generators on the right to be those automorphisms of $\Sigma^{\X G}$ induced by right multiplication by elements of $g$. Tracing through the definitions a final time, we see that $\alpha_0$ maps the one set of generators to the other. This completes the proof.
\end{proof}

We have the functor $G \to A^{eff}(\FfG)$ that takes, for example, a 2-simplex $(g, h)$ to the diagram
\[\begin{tikzcd}[column sep = small]
&& G \ar[equals]{dl} \ar{dr}{r_g} \\
& G \ar[equals]{dl} \ar{dr}{r_g} & & G \ar[equals]{dl} \ar{dr}{r_h} \\
G && G && G.
\end{tikzcd}\]
We have the composite
\[i_G: G \X \F_* \to A^{eff}(\FfG) \X A^{eff}(\F) \to A^{eff}(\FfG).\]

\begin{prop} \label{roff}
$i_G$ is the universal commutative monoid with $G$-action, which is to say the initial functor satisfying the Segal condition from $G \X \F_*$ to a category with finite products.
\end{prop}

\begin{proof}
First we must show that $i_G$ is indeed a commutative monoid. $G$ doesn't make any difference here; we just need to show that $i : \F_* \to A^{eff}(\F)$ satisfies the Segal condition. This follows from the fact that the image of the inert map $\chi_j : \angs{n} \to \angs{1}$ is the span

\[
\begin{tikzcd}[column sep=small]
& {[1]} \ar{ld}{j} \ar{rd} & \\
{[n]} && {[1]}.
\end{tikzcd}
\]

Now let $U(G)$ be the universal category supporting a commutative monoid with $G$-action. Since $U(G)$ is the initial category under $G \X \F_*$ that takes certain diagrams to limit diagrams, Proposition 5.3.6.2 of \cite{HTT} gives a prescription for building it as the opposite of a full subcategory of a certain localization $S^{-1}\mb{Psh}((G \X \F_*)^\op)$ of the presheaf category $\mb{Psh}((G \X \F_*)^\op)$. Since we're in this business, let's let $Y: (G \X \F_*)^\op \to \mb{Psh}((G \X \F_*)^\op)$ denote the Yoneda embedding.

In this case, the localization $S$ is generated by the morphisms
\[\coprod_{\angs{n}^o} Y(\angs{1}) \to Y(\angs{n})\]
given by precomposition with the inert maps, and so localization is ``Segalification" and the local objects are exactly the commutative monoid spaces with $G$-action. Thus $U(G)^\op$ is the full subcategory of $\Fun(\F_* \X G, \Top)$ spanned by the Segalifications of
\[\coprod_{\angs{n}^o} Y(\angs{1})\]
as $n$ varies. But $Y(\angs{1})$ is, by definition, left Kan extended along the inclusion 
\[\angs{1} : \ast \to \F_* \X G\]
and so its Segalification is the free commutative monoid with $G$-action on one generator. Since Segalification preserves coproducts, the other objects follow.

Now we've given an equivalence between $U(G)^\op$ and $\mc{L}(G)$, and therefore \newline $A^{eff}(\FfG)$. We know that this category is canonically self-opposite, so we might as well forget the $\op$ on $U(G)$. Let's show that this equivalence comes from $i_G$.

Since $A^{eff}(\FfG)$ is semiadditive, specifying a commutative monoid with $G$-action $BG \X \F_* \to A^{eff}(\FfG)$ is equivalent to specifying it on $G \X \angs{1}$ (see Corollary 2.4.3.10 of \cite{HA}). Both $i_G$ and the universal commutative monoid in $U(G)$ take $G \X \angs{1}$ to the $G$-set $G$ with its right action on itself. This completes the proof.  \end{proof}

\begin{cor} \label{riff}
Let $\mb{C}$ be an $\iy$-category with finite products. Then pullback along $i_G$ gives an equivalence
\[\Fun^\X(A^{eff}(\FfG), \mb{C}) \to \CMon(\Fun(G, \mb{C})).\]
\end{cor}
\begin{proof}
This is just a restatement of \ref{roff}.
\end{proof}

We deduce Theorem \ref{thm:Gadd} as the special case of \ref{riff} where $\mb{C}$ is semiadditive.

\section{The proof of Lemma \ref{thm1DA}} \label{app:lem}
%\begin{lem}
%Suppose $S$ is an $\iy$-category, $p : C \to S$ is an inner fibration and $q : S \to C$ is a section of $p$ such that for every object $s \in S$, $q(s)$ is an initial object of $C_s$. Then $q$ is left anodyne.
%\end{lem}
%\begin{proof}
%By induction on the simplices of $S$, we reduce to the following:
%\begin{lem}
%Suppose $p : C \to \D^n$ is an inner fibration, and $q : \D^n \to C$ is a section of $p$ such that for every vertex $x \in \D^n$, $q(x)$ is an initial object of $C_x$. Then the inclusion
%\[C_{\del \D^n} \amalg_{\del \D^n} \D^n \to C\]
%is left anodyne.
%\end{lem}
%\end{proof}
\begin{lem}\label{lem:C_XD}
Let $\mb{C}$ be an $\iy$-category, $z : \mb{D} \to \mb{C}$ the inclusion of a full subcategory and $X \in \mb{C}$ an object. Let $\star$ denote the join of simplicial sets, and let 
\begin{align*} i_0 : \D^0 \to \D^0 \star (\D^n \X \D^1) \\  i_1 : \D^n \X \{0\} \to \D^0 \star (\D^n \X \D^1) \\ i_2 : \D^n \X \{1\} \to \D^0 \star (\D^n \X \D^1)\end{align*}
be the natural inclusions.
We define a simplicial set $\mb{C}_{X/\mb{D}/}$ whose $n$-simplices are maps $\alpha : \D^0  \star (\D^n \X \D^1) \to \mb{C}$ with 
\[\alpha \circ i_0 = X, \, \, \, \, \alpha \circ i_1 \in \Fun(\D^n, \mb{D}).\]
Then the map $p_2 : \mb{C}_{X/ \mb{D}/} \to \mb{C}$ coming from precomposition with $i_2$ is cocartesian, and its cocartesian edges are those $\alpha$ for which the image of $ \alpha \circ i_1$ is an equivalence. Moreover, the inclusion
\[\lambda : \mb{C}_{X/} \X_{\mb{C}} \mb{D} \inj \mb{C}_{X / \mb{D}/}\]
formed by precomposition with the collapse map $\D^0 \star (\D^n \X \D^1) \to \D^0 \star \D^n$ is coinitial, and therefore left anodyne \cite[Proposition 4.1.1.3]{HTT}.
\end{lem}
\begin{proof}
First we show that any edge $\alpha$ for which $\alpha \circ i_1$ is an equivalence is cocartesian. This is the claim that any commutative diagram of the form
\[\begin{tikzcd}[column sep = tiny]
&& T_2  \ar{dd} \\
& T_1 \ar{ur}{\sim} \ar[crossing over]{rr} \ar{dd} && T_3 \ar{dd} \\
X \ar{ur} \ar{dr} && W_2 \ar{dr} \\
& W_1 \ar{ur} \ar{rr} && W_3 \end{tikzcd}\]
with $T_1 \to T_2$ an equivalence can be completed to a diagram from $\D^0 \star (\D^2 \X \D^1)$, which is clear by inspection. Since there are plenty of these edges, $p_2$ is cocartesian.

Now we tackle the coinitiality claim. In fact, we'll show that $\lambda$ admits a right adjoint, which suffices. Let $\La \to \D^1$ be the cocartesian fibration classified by $\la$; an $n$-simplex of $\La$ is a map $\tau : \D^n \to \D^1$ together with a map
\[\alpha : \left( \D^0 \star \left \{(i, j) \in \D^n \X \D^1 \, | \, j = 0 \text{ or } i \in \tau^{-1}(1)\right \} \right) \to \mb{C}\]
We wish to show that $\alpha$ is also cartesian. In fact, we claim that an edge
\[\begin{tikzcd}[column sep = tiny]
& T_2 \ar{rr} && T_3 \ar{dd} \\
X \ar{ur} \\
&&& W_3 \end{tikzcd}\]
of $\Lambda$ over the nondegenerate edge of $\D^1$ is cartesian if $T_2 \to T_3$ is an equivalence. This is the claim that any commuting diagram of the form
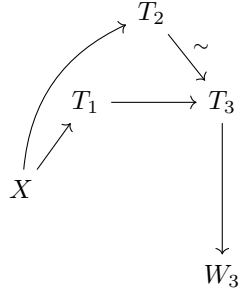
\[\begin{tikzcd}[column sep = tiny]
&& T_2 \ar{dr}{\sim} \\
& T_1 \ar{rr} && T_3 \ar{dd} \\
X \ar{ur} \ar[bend left]{uurr} \\
&&& W_3 \end{tikzcd}\]
can be extended to a commuting diagram of the form
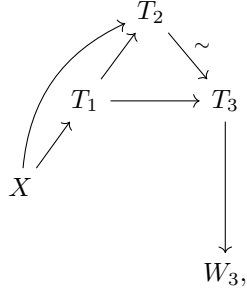
\[\begin{tikzcd}[column sep = tiny]
&& T_2 \ar{dr}{\sim} \\
& T_1 \ar{ur} \ar{rr} && T_3 \ar{dd} \\
X \ar{ur} \ar[bend left]{uurr} \\
&&& W_3, \end{tikzcd}\]
which, again, is clear.
\end{proof}
We list some formal consequences of Lemma \ref{lem:C_XD}.
\begin{cor} \label{cor:reskanval}
Let $\mb{C}_{X/ \mb{D}/}^\Top$ be a fibrant replacement for $\mb{C}_{X/ \mb{D}/}$ in the covariant model structure over $\mb{C}$, so that $\mb{C}_{X/ \mb{D}/}^\Top \to \mb{C}$ is a left fibration and for each object $W \in \mb{C}$, the map
\[(\mb{C}_{X/ \mb{D}/}^\Top)_W \to (\mb{C}_{X/ \mb{D}/})_W\]
is a Kan-Quillen weak equivalence. Then the functor classified by $\mb{C}_{X/ \mb{D}/}^\Top$ is equivalent to the restriction and left Kan extension $z_! z^* \Map(X, -)$ of the functor corepresented by $X$.

Now let $\beta : \mb{C}_{X/ \mb{D}/} \to \mb{C}_{X/}$ be given on n-simplices by precomposition with 
\[i_0 \star i_2 : \D^0 \star \D^n \to \D^0 \star (\D^n \X \D^1).\]
 Since $\mb{C}_{X/}$ is a left fibration, we have a commutative diagram
\[\begin{tikzcd}
\mb{C}_{X/} \X_\mb{C} \mb{D} \ar{d}{\lambda} \ar{rr}{\beta} & & \mb{C}_{X/} \ar{d} \\
\mb{C}_{X/ \mb{D}/}  \ar{r} & \mb{C}_{X/\mb{D}/}^\Top \ar{r} \ar[dotted]{ur}{\beta'} & \mb{C}.
\end{tikzcd} \]
It follows that the counit map $c_X : z_! z^* \Map(X, -) \to \Map(X, -)$ is given, after unstraightening, by $\beta'$.
\end{cor}

\begin{lem}\label{lem:Kantyp}
Let $W \in \mb{C}$ and let $f : X \to W$ be a morphism in $\mb{C}$. Then the homotopy fiber of
\[c_{X, W} : (z_! z^* \Map(X, -))_W \to \Map(X, W)\]
over $f$ is given, up to weak equivalence, by the simplicial set $\mho_{W, f}$ whose $n$-simplices are maps $Z : \D^{n + 2} \to \mb{C}$ such that 
\begin{itemize}
\item $Z|_{\D^{\{0, n + 2\}}} = f$, and
\item for each $i$ with $0 < i < n + 2$, $Z(i) \in \mb{D}$.
\end{itemize}
\end{lem}
\begin{proof}
We know that the Joyal model structure is self-enriched, by using \cite[Corollary 2.2.5.4]{HTT} to deduce that the pushout-product of a trivial cofibration with a cofibration is a trivial cofibration, and it follows that if $K \to L$ is any cofibration of simplicial sets and $\mb{E}$ is a quasicategory, then $\Fun(L, \mb{E}) \to \Fun(K, \mb{E})$ is a categorical fibration. Since $\beta$ is formed from such a fibration by pullback, $\beta$ is also a categorical fibration.

The value of $c_X$ on $W$ is given, up to weak equivalence, by the map
\[\beta_W : (\mb{C}_{X/\mb{D}/})_W \to (\mb{C}_{X/})_W = \Hom^L(X, W).\]
Since the target of $\beta_W$ is a Kan complex and $\beta_W$ is a categorical fibration, it is a cocartesian fibration \cite[Proposition 3.3.18]{HTT}, and since a fibrant replacement for $\beta_W$ in the covariant model structure over $\Hom^L(X, W)$ is automatically a Kan fibration, the fibers of $\beta_W$ are its homotopy fibers.

By definition, the fiber $\beta_{W, f}$ over $f \in \Hom^L(X, W)$ is the simplicial set whose $n$-simplices are maps $Z' : (\D^0 \star (\D^n \X \D^1)) / (\D^n \X \{1\}) \to \mb{C}$ such that 
\begin{itemize}
\item $Z' \circ i_1 \inc \mb{D}$, and
\item the $(n + 1)$-simplex $Z' \circ (i_0 \star i_2)$ is the image of $f $ under the rightmost degeneracy; that is, it is the totally degenerate $n$-simplex of $\mb{C}_{X/}$ at $f$.
\end{itemize}
In other words, an $n$-simplex of $\beta_{W, f}$ is a map 
\[Z'' : \D^0 \star ((\D^n \X \D^1) / (\D^n \X \{1\}))\]
with $Z''(\D^n \X \{0\}) \in \mb{D}_n$ and $Z''(\D^0 \star ((\D^n \X \{1\})/(\D^n \X \{1\})) = f$. From here, the proof that $\beta_{W, f} \simeq \mho_{W, f}$ is a minor variant of the proof of \cite[Proposition 4.2.1.5]{HTT}.
\end{proof}

We now immerse ourselves in the notation of Lemma \ref{thm1DA}, which we restate here for convenience.
\begin{lem} 
Let $\mc{M}$ be an epiorbital category, $\mc{N}$ a downwardly closed subcategory of $\mc{M}$ and $\mc{T}$ its upwardly closed complement. Denote the restriction $ A^{eff}(i_\mc{T})^*$ and the left Kan extension $A^{eff}(i_\mc{T})_!$ respectively by $\Pi^\mc{T}$ and $\Gamma^\mc{T}$, and similarly denote $A^{eff}(j_\mc{N})^*$ and $A^{eff}(j_\mc{N})_!$ by $\Xi^\mc{N}$ and $\Phi^\mc{N}$ respectively. Then there's a cofiber sequence of functors $\Mack(\mc{M}, \CMon) \to \Mack(\mc{M}, \CMon)$
\[\Gamma^\mc{T} \Pi^\mc{T} \os{\epsilon} \to \Id  \os{\eta} \to \Xi^\mc{N} \Phi^\mc{N},\]
where $\epsilon$ and $\eta$ are the counit and unit of their respective adjunctions.
\end{lem}
\begin{proof}
Since all of the functors in this sequence are colimit-preserving, it suffices to check that its value on each corepresentable Mackey functor is a cofiber sequence. Let $X \in A^{eff}(\mc{M})$ be an object and let $f : X \lef Y \to W$ be a morphism in $A^{eff}(\mc{M})$. We'll analyze the fiber $\mho_{W, f}$ of Lemma \ref{lem:Kantyp}.

For $\mb{C}$ an $\iy$-category, let $\twa_\mb{C}$ be the twisted arrow category of $\mb{C}$ (see \cite[\S 2]{Bar14}) and let $\twa^\mb{C} := \twa_\mb{C}^\op$ be its opposite. Then an $n$-simplex of $\mho_{W, f}$ is by definition a functor 
\[\varrho : \twa^{\D^{n + 2}} \to \mc{M}^\amalg\]
such the span 
\[\varrho(0, 0) \lef \varrho(0, n + 2) \to \varrho(n+ 2, n+ 2)\]
 coincides with $f$ and $\varrho(i, i) \in \mc{T}^\amalg$ for all $i$ with $0 < i < n + 2$. By the upward closedness of $\mc{T}$, the latter condition implices that $\varrho(i,j) \in \mc{T}^\amalg$ for all $(i, j) \neq (0, 0), (n + 2, n +2)$, and in particular $\mho_{W, f}$ is empty unless $Y \in \mc{T}^\amalg$. We claim that if $Y \in \mc{T}^\amalg$, then $\mho_{W, f}$ is contractible. In fact, let $\Gamma_{W, f}$ be the subsimplicial set of $\mho_{W, f}$ whose $n$-simplices are those which factor through the morphism 
\[\gamma : \twa^{\D^{n + 2}} \to \twa^{\D^{n + 2}} \hspace{4 em}  \gamma(i, j) = \begin{cases}
(0, 0) & \text{if } (i, j) = (0, 0) \\
(n + 2, n + 2) & \text{if } (i, j) = (n + 2, n + 2) \\
(0, n + 2) & \text{otherwise.} \end{cases}  \]

If $Y \in \mc{T}^\amalg$, then clearly $\Gamma_{W, f} \cong *$, and we claim that $\Gamma_{W, f}$ is a simplicial deformation retract of $\mho_{W, f}$. We'll do this in two stages as follows. For each integer $k$, let $\gamma^L_k : \twa^{\D^{n + 2}} \to \twa^{\D^{n + 2}}$ be defined by
\[\gamma^L_k (i, j) = (\min(i-k, 0), j)\]
and dually, define $\gamma^R_k : \twa^{\D^{n + 2}} \to \twa^{\D^{n + 2}}$ by
\[\gamma^R_k(i, j) = (i, \max(j + k, n + 2)).\]
Let $\Gamma^L_{W, f}$ be the subsimplicial set of $\mho_{W, f}$ whose $n$-simplices factor through $\gamma^L_{n + 1}$, and define $\Gamma^R_{W, f}$ similarly; we will show that each of $\Gamma^L_{W, f}$ and $\Gamma^R_{W, f}$ is a simplicial deformation retract of $\mho_{W, f}$, and since
\[\Gamma_{W, f} = \Gamma^L_{W, f} \cap \Gamma^R_{W, f},\]
this will complete the proof of the claim. We will prove the result for $\Gamma^R_{W, f}$; the result for $\Gamma^L_{W, f}$ is, of course, entirely dual.

For each $n, l$ with $0 \leq l \leq n + 1$, let $\tau_{n, l} : \D^n \to \D^1$ be the unique map with
\[\tau_{n, l}^{-1}(0) = [0, \cdots, n - l],\]
where we interpret $[0, -1]$ as the empty interval. Then we define a map
\[\Theta : \mho_{W, f} \X \D^1 \to \mho_{W, f}\]
on $n$-simplices by
\[(\varrho, \tau_{n, l}) \mapsto \varrho \circ \gamma^R_l.\]
Then $\Theta_1$ is a retraction onto $\Gamma^R_{W, f}$, so we have the required simplicial homotopy.

What we have proved so far is that 
\[\eps_{c_X} : \Gamma^\mc{T} \Pi^\mc{T} c_X \to c_X,\]
after evaluation on an object $W$, is homotopic to the inclusion of the connected components of $\Map(X, W)$ comprising those maps $f : X \lef Y \to W$ with $Y \in A^{eff}(\mc{T})$. What's left is easy: $\eta_{c_X}(W)$ is the natural map
\[\Map_{A^{eff}(\mc{M})}(X, W) \to \Map_{A^{eff}(\mc{N})}(j(X), j(W)),\]
which is just the projection away from the image of $\eps_{c_X}(W)$. This concludes the proof of Lemma \ref{thm1DA}.
\end{proof}
\bibliographystyle{alpha}
\bibliography{mybib}
\end{document}